\documentclass[a4paper]{amsart}
\usepackage[utf8]{inputenc}
\usepackage[T1]{fontenc}

\usepackage{times}

\usepackage[pagebackref,colorlinks=true,pdfpagemode=none,urlcolor=blue,
linkcolor=blue,citecolor=blue]{hyperref}

\usepackage{amsmath,amsfonts,amssymb,amsthm}
\usepackage{amsthm, amssymb, amsmath, amsfonts, mathrsfs}

\usepackage{mathrsfs}
\usepackage{MnSymbol}
\usepackage{scalerel} 

\usepackage{color}
\usepackage{accents}

\usepackage{bbm}

\usepackage[normalem]{ulem}

\newtheorem{theorem}{Theorem}[section]
\newtheorem{lemma}[theorem]{Lemma}

\newtheorem{corollary}[theorem]{Corollary}
\newtheorem{proposition}[theorem]{Proposition}

\theoremstyle{remark}
\newtheorem{remark}[theorem]{Remark}

\renewenvironment{proof}[1][Proof]{ {\itshape \noindent {#1.}} }{$\Box$
\medskip}

\numberwithin{equation}{section}
\newcommand{\R}{\mathbb{R}}

\newcommand{\Z}{\mathbb{Z}}

\newcommand{\Pb}{\mathbb{P}}
\newcommand{\PP}{\mathbf{P}}
\newcommand{\E}{\mathbb{E}}
\newcommand{\bbQ}{\mathbb{Q}}

\newcommand{\F}{\mathcal{F}}

\newcommand{\G}{\mathcal{G}}

\newcommand{\W}{\mathscr{W}}

\newcommand{\D}{\mathcal{D}}
\newcommand{\V}{\mathcal{V}}

\newcommand{\eps}{\varepsilon}

\newcommand{\cE}{\mathcal{E}}

\newcommand{\EE}{\mathbf{E}}

\renewcommand{\d}{\mathrm{d}}
\def\blue{\textcolor{black}}

\newcommand{\x}{\mathbf{x}}
\newcommand{\bT}{\mathbb{T}}
\newcommand{\cZ}{\mathcal{Z}}

\newcommand{\cX}{\mathcal{X}}
\newcommand{\cY}{\mathcal{Y}}
\newcommand{\cW}{\mathcal{W}}

\newcommand{\rhof}{\rho_{\mathrm{f}}}
\newcommand{\rhob}{\rho_{\mathrm{b}}}
\newcommand{\lf}{\lfloor}
\newcommand{\rf}{\rfloor}
\newcommand{\rhoa}{\rho_{\mathrm{app},\mathrm{m}}}
\newcommand{\tildegapp}{\tilde{g}_{\mathrm{app}}}
\newcommand{\tildeJapp}{\tilde{J}_{W,\mathrm{app}}}
\newcommand{\cal}{\mathcal}

\begin{document}
\title{Effective diffusivities in periodic KPZ}

\author{Yu Gu, Tomasz Komorowski}

\address[Yu Gu]{Department of Mathematics, University of Maryland, College Park, MD 20742, USA. }
\email{yugull05@gmail.com}

\address[Tomasz Komorowski]{Institute of Mathematics, Polish Academy of Sciences, ul.
Śniadeckich 8, 00-636 Warsaw, Poland. }
\email{tkomorowski@impan.pl}

\maketitle

\begin{abstract}
For the KPZ equation on a torus with a $1+1$ spacetime white noise, it was shown in \cite{GK21,ADYGTK22} that the height function satisfies a central limit theorem, and the variance can be written as the expectation of an exponential functional of Brownian bridges. In this paper, we consider another physically relevant quantity, the winding number of the  directed polymer on a cylinder, or equivalently, the displacement of the directed polymer endpoint    in a spatially periodic random environment. It was shown in \cite{YGTK22} that the polymer endpoint   satisfies a central limit theorem on diffusive scales. The main result of this paper is  an explicit expression of the effective diffusivity, in terms of the expectation of another exponential functional of Brownian bridges. Our argument is based on a combination of tools from Malliavin calculus, homogenization, and diffusion in distribution-valued random environments.


\medskip


\end{abstract}


\section{Introduction}

\subsection{Main result}
We are interested in the stochastic heat equation (SHE) with a spatially periodic space-time white noise 
\begin{equation}\label{e.she}
\begin{aligned}
&\partial_t Z=\tfrac12\Delta Z+\beta Z\xi, \quad\quad t>0, x\in\R,\\
&Z(0,x)=\delta(x).
\end{aligned}
\end{equation}
Here $\xi$ is a generalized Gaussian random field over $\R\times\R$ with the covariance function 
\[
\EE[\xi(t,x)\xi(s,y)]=\delta(t-s)\sum_{n\in\Z}\delta(x-y+n).
\]
We assume that $\xi$ is built on a probability space $(\Omega,\F,\PP)$
and view it as a spacetime white noise on $\R\times \bT$ that is
periodically extended  to $\R\times \R$. Here $\bT$ is the 
unit torus, defined as the interval $[0,1]$ with identified endpoints.
The product between $Z$ and $\xi$ is interpreted in the
It\^o-Walsh sense. 

The random function $Z$ is associated with the model of a directed
polymer in a random environment and the parameter $\beta>0$ plays the
role of the inverse temperature. There are two important physical
quantities: the free energy and the endpoint displacement of the polymer. Define
$\bar{Z}_t=\int_\R Z(t,x)dx$ as the point-to-line partition function. It was known that $\log \bar{Z}_t$ satisfies a central limit theorem, with the drift and the variance described explicitly in terms of some auxiliary Brownian bridges, see \cite[Theorem 1.1]{GK21}, \cite[Proposition 4.1]{GK211} and \cite[Eq. (2.10), (4.2)]{ADYGTK22}:
\[
\frac{\log \bar{Z}_t+\gamma(\beta)t}{\sqrt{t}}\Rightarrow
N(0,\Sigma^2(\beta)),\quad \mbox{as }t\to\infty,
\]
where $\gamma(\beta),\Sigma^2(\beta)>0$ are constants given by 
\begin{equation}\label{e.gammaSigma}
\begin{aligned}
&\gamma(\beta)=\frac{\beta^2}{2}\E_{W_1}\left[\frac{1}{\left(\int_{0}^1e^{\beta
    W_1(y)}dy\right)^{2}}\right]=\frac{\beta^2}{2}+\frac{\beta^4}{24},\\
&\Sigma^2(\beta)= \beta^{2}\E_{W_1}\left[\left(\E_{W_2}\frac{1}{\int_0^1 e^{\beta(W_1(y)+W_2(y))}dy}\right)^2\right].
\end{aligned}
\end{equation}
Here  $W_1,W_2$ are independent Brownian bridges \blue{connecting} $(0,0)$ and $(1,0)$, and throughout the paper, we use $\E_{W_i}$ as the expectation on $W_i$.

Define 
\begin{equation}\label{e.defrhoIntro}
\rho(t,x)=\frac{Z(t,x)}{\bar{Z}_t},\quad x\in\R,
\end{equation} which is the quenched endpoint density of the
associated directed polymer. It was shown in \cite{YGTK22} that there
exists some $\sigma^2(\beta)>0$ such that for any $f$ belonging to
$C_b(\R)$ - the space of bounded continuous functions on $\R$ - we have 
\begin{equation}\label{e.diffusivityendpoint}
\EE \int_{\R} f(\frac{x}{\sqrt{t}})\rho(t,x)dx\to \int_{\R} f(x) \frac{1}{\sqrt{2\pi \sigma^2(\beta)}}\exp(-\frac{x^2}{2\sigma^2(\beta)})dx, 
\end{equation}
as $t\to\infty$. In other words, the annealed endpoint distribution converges under the diffusive scaling to a centered Gaussian with variance $\sigma^2(\beta)$.

The main result of the paper is to derive an explicit expression of
$\sigma^2(\beta)$:



\begin{theorem}\label{t.mainth}
Suppose that $\{W_i\}_{i=1,2,3}$ are  three independent Brownian
bridges connecting $(0,0)$ and $(1,0)$. Then, the effective diffusivity $\sigma^2(\beta)$ admits the following expression: 
\begin{equation}\label{e.exsigma}
\sigma^2(\beta)=1+\beta^2\E_{W_1}\E_{W_3}\big[\mathcal{A}(\beta,W_1,W_3)^2\big],
\end{equation}
where
\[
 \mathcal{A}(\beta,W_1,W_3)
    =\int_{\bT^2} \Xi(\beta,y,W_1)
    \left(\frac{e^{\beta W_1(z)+\beta W_3(z)}}{\int_{\bT}
    e^{\beta W_1(z')+\beta
      W_3(z')}dz'}-1\right)1_{[0,y]}(z)dydz,
      \]
      with
      \[
\Xi(\beta,y,W_1)=\E_{W_2}\left[\frac{ e^{\beta W_2(y) -\beta
    W_1(y)}}{\Big(\int_{\bT}e^{\beta W_2(y')}e^{-\beta W_1(y')}dy'\Big)^2}\right].
\]
\end{theorem}

\subsection{Motivation}
There is an intriguing  relation between the two diffusion constants $\Sigma^2(\beta)$ and $\sigma^2(\beta)$, predicted by \'Eric Brunet through the replica method, 
see \cite[Eq. (20)]{Bru03}. We translate it in our notations\footnote{The $\gamma$ parameter in \cite[Eq. (20)]{Bru03} is $\beta^2$ in our notation.}:
\begin{equation}\label{e.brunet}
\sigma^2(\beta)=\frac{2}{\beta^2}\Sigma^2(\beta)-\frac{1}{2\beta}\frac{d}{d\beta} \Sigma^2(\beta)=-\frac{\beta^3}{2}\frac{d}{d\beta}\Big(\frac{\Sigma^2(\beta)}{\beta^{4}}\Big), \quad\quad \beta>0.
\end{equation}

In order to apply the replica method   to study the fluctuations of the free energy $\log \bar{Z}_t$, it reduces to computing the ground state energy of the $n-$particle Delta Bose gas 
\[
\mathcal{H}_n=\frac12\sum_{i=1}^n \nabla_{x_i}^2+\beta^2\sum_{1\leq i<j\leq n} \delta(x_i-x_j)
\] on the torus (with periodic boundary condition), for all $n\geq1$. This was done in \cite{BD00,BD001} by solving the Bethe Ansatz equation. The quenched limiting free energy $\gamma(\beta)$  in \eqref{e.gammaSigma} matches the prediction given by the replica method, see  \cite[Eq. (13)]{Bru03}. There is also a predicted integral form for $\Sigma^2(\beta)$, see \cite[Eq. (14)]{Bru03}, but we do not know how to compute the expectation in the expression of $\Sigma^2(\beta)$ in \eqref{e.gammaSigma}. 

Meanwhile, it is more complicated to compute the effective diffusivity of the polymer endpoint  using the replica method, and it turns out that one needs to find the ground state energy of the same Hamiltonian but with a different boundary condition, see \cite[Section 3.1]{Bru03}. Through solving a new Bethe Ansatz equation, Brunet found a relation between the ground state energies of the two Hamiltonians with different boundary conditions, through which he derived \eqref{e.brunet}.

It was our goal to justify \eqref{e.brunet} rigorously, and to understand why there exists such a relation between the two (non-universal) diffusion constants. With Theorem~\ref{t.mainth}, it seems that we are  halfway there, since both $\sigma^2(\beta)$ and $\Sigma^2(\beta)$ are now written as the expectations of exponential functionals of independent Brownian bridges. Nevertheless, to prove \eqref{e.brunet} using the expressions in \eqref{e.gammaSigma} and \eqref{e.exsigma} seems to be a highly nontrivial problem on Brownian bridges, which we choose not to pursue here. Using \eqref{e.gammaSigma} and \eqref{e.exsigma}, one can perform a small $\beta$ expansion, which indeed matches the predictions in \cite{Bru03}, and suggests the validity of \eqref{e.brunet}:
\begin{corollary}
For $\beta\ll1$, we have
\begin{equation}\label{e.betaexp}
\begin{aligned}
&\sigma^2(\beta)=1+\tfrac{\beta^4}{360}+O(\beta^6),\\
&\Sigma^2(\beta)=\beta^2+\tfrac{\beta^4}{12}-\tfrac{\beta^6}{360}+O(\beta^8).
\end{aligned}
\end{equation}
\end{corollary}

Another reason why one may be interested in explicit expressions of the diffusion constants comes from the attempt to understand the $1:2:3$ scaling in the $1+1$ KPZ universality class. In this paper, the length of the torus was fixed to be $L=1$. If we denote the diffusion constants by $\Sigma^2(\beta,L)$ and $\sigma^2(\beta,L)$ for general $L>0$, it is straightforward to derive the following relations through a rescaling:
\[
\begin{aligned}
&\Sigma^2(\beta, 1)=L^2\Sigma^2(\frac{\beta}{\sqrt{L}},L),\\
&\sigma^2(\beta,1)=\sigma^2(\frac{\beta}{\sqrt{L}},L).
\end{aligned}
\]
Therefore, the asymptotic behaviors of $\Sigma^2(\beta,1)$ and $\sigma^2(\beta,1)$ as $\beta\to\infty$  translate into the behaviors of $\Sigma^2(1,L)$ and $\sigma^2(1,L)$ as $L\to\infty$. For $L=\infty$ with the problem posed on the whole line, it is known that the free energy and the polymer endpoint behave sub- and super-diffusively with the $1/3$ and $2/3$ exponents, so one expects that $\Sigma^2(1,\infty)=0$ and $\sigma^2(1,\infty)=\infty$. As a matter of fact, the decaying rate of $\Sigma^2(1,L)\downarrow0$ and the growing rate of $\sigma^2(1,L)\uparrow\infty$ as $L\to\infty$  is related to the $1:2:3$ scaling: 

(i) For the height function (free energy), the critical scale comes from the balance of the one-point variance $t\Sigma^2(1,L)$ and the transversal roughness $L$. \blue{The transversal roughness here refers to the fluctuations of $h(t,x)-h(t,0)$ at stationarity, where $h$ solves the KPZ equation on the torus. It was shown in \cite[Theorem 4.1]{ADYGTK22} that, as $L\to\infty$, the variance $\Sigma^2(1,L)$ decays like $L^{-1/2}$, which suggests the $1:2:3$ scaling and also leads to the optimal variance of the height function in certain regimes with $t,L\to\infty$ together, see \cite[Theorem 1.1]{ADYGTK22}. The heuristic is that, suppose $t\Sigma^2(1,L)\asymp tL^{-1/2}\gg L$, then the transversal roughness is much smaller than the one-point fluctuation, and if we write $h(t,x)=h(t,0)+(h(t,x)-h(t,0))$, then the second term on the r.h.s. has a much smaller fluctuation than the first term, so that $h(t,x)$ and $h(t,0)$ are almost perfectly correlated for $t\gg1$. Only when $tL^{-1/2}$  is of order $O(L)$, we have the nontrivial correlation kicks in, and this gives $L\asymp t^{2/3}$.}

(ii) For the polymer endpoint, the critical scale comes from the balance of the diffusive standard deviation $\sqrt{t\sigma^2(1,L)}$ and the length of the cell $L$. It is expected and also predicted by the replica method that $\sigma^2(1,L)\asymp L^{1/2}$, which suggests the $1:2:3$ scaling. \blue{Roughly speaking, in the region of $\sqrt{t\sigma^2(1,L)}\asymp \sqrt{tL^{1/2}}\gg L$, the polymer path visits many cells and one still expects to see the homogenization
phenomenon and the central limit theorem, thus, it is natural to guess the critical length scale comes from the relation $\sqrt{tL^{1/2}}\asymp L$, which leads to $L\asymp t^{2/3}$.} Starting from the expression \eqref{e.exsigma}, one may try to show the equivalent statement of  $\sigma^2(\beta,1)\asymp \beta$ as $\beta\to\infty$, although we do not pursue it here.

In the end, let us point out that, in order to prove the asymptotic
behaviors of the diffusion constants, it simplifies greatly if one
could derive some explicit formulas first. For general homogenization
type problems, the effective diffusivity is typically written in an
abstract form, in some cases through the solution to a cell problem,
the so-called corrector, in some other cases in terms of the
Green-Kubo formula or the solution to a certain variational
problem. Therefore, it was actually a surprise to us that the
$\sigma^2(\beta)$ (and $\Sigma^2(\beta)$) may be expressed explicitly
in terms of the underlying invariant measure for the KPZ equation on
the unit torus, which is the Brownian bridge. In a sense, the corrector is explicit in our problem, which we will discuss in more detail later. As mentioned previously, the central limit theorem was already derived in \cite{YGTK22}, with $\sigma^2(\beta)$ written as the integral of an abstract covariance function, but it seems hopeless to us to do any quantitative analysis from that expression. 

\subsection{Context}
The study of the KPZ equation and the $1+1$ KPZ universality class has
witnessed  remarkable progress during the past years. We will not
attempt to survey the huge literature here, and refer the readers to
the reviews \cite{Cor12,Qua12,QS15} and the references therein. For a nice introduction to the SHE, we refer to the monograph \cite{DK14}. Our problem is more related to those in a bounded region since the compactness makes a huge difference concerning the large scale random fluctuations. For (totally) asymmetric simple exclusion processes on a finite line segment with periodic or open boundaries, exact diffusion constants for the particle current were derived in \cite{DEM93,DEM95,DM97}. More recently, there is a series of work \cite{BL16,BL18,BL19,BL21,BLS20,Liu18} on the regimes where the size of the torus and the time go to infinity simultaneously, and the fluctuations of the height function were derived in different cases. For the open KPZ equation, the one on an interval with Neumann boundary conditions, a lot of recent progress has been on the construction of explicit invariant measures \cite{BCY23,BD21,BKWW21,Cor22,CK21}.

The problem we consider here can also be considered as
the homogenization of a diffusion in a random environment. Although
the directed polymer is almost always formulated in the form of 
Gibbs measures of paths in a random environment, it can also be viewed
as a passive scalar problem, with the drift given by the solution to a
stochastic Burgers equation. This perspective plays a very important
role in our analysis: for the polymer endpoint, one can decompose its
total variance into two parts, the variance of the quenched mean and
the mean of the quenched variance, which, roughly speaking, correspond to the disorder variance and the thermal variance respectively. It turns out that the main
difficulty comes from analyzing the quenched mean, which is actually
the corrector if we put it in the context of homogenization, see
Section~\ref{s.sketch} for more discussion. For our case of a
spacetime white noise, the solution to the stochastic Burgers equation
is not function-valued, so one needs to deal with the so-called
singular diffusion. Luckily for us, there has been a lot of recent
progress on the study of diffusion with (very singular) distributional
drifts \cite{CC18,DD16,MH13,HZZZ21,KP22}, which was inspired by the
breakthrough in the study of  singular SPDEs, using tools such as rough paths \cite{MH13}, 
regularity structures \cite{MH14} and paracontrolled calculus
\cite{GIP15}, and we will borrow tools from there.

Another important ingredient we rely on to derive the explicit
expressions of the diffusion constants is the so-called Clark-Ocone
formula from Malliavin calculus. It involves computing the Malliavin
derivative of the free energy or the quenched mean of the polymer
endpoint with respect to the underlying white noise, and rewriting
their fluctuations in the form of It\^o integrals. The key   here is
the \emph{It\^o integral}, which can be viewed as the sum of
martingale differences, so that the second moment can be computed in a
rather straightforward way. To compare with the homogenization
argument, one should note that, from a probabilistic perspective, the whole point of constructing the
corrector   through solving a certain Poisson
equation is to extract the martingale part from the drift, which
contributes to the effective diffusivity. In a certain sense, the
Clark-Ocone formula provides a different way of performing the
semi-martingale decomposition, and in some cases it leads to ``new''
formulas like \eqref{e.gammaSigma} and \eqref{e.exsigma}. For the
directed polymer, there is a simple yet crucial property that the
derivative of the free energy with respect to the noise is given by
the quenched density of the polymer path. Compared to other
homogenization problems, it is this special property that makes the
Clark-Ocone formula useful in this case.

In order to study the super-diffusive behaviors in turbulent transport, it is natural to analyze how the effective diffusivity diverges as the domain in consideration becomes larger and larger. In \cite{AF98}, the so-called box diffusivity was defined, and its asymptotics was analyzed to derive the relevant scaling exponents of several models. More recently, for the model of a Brownian particle in the curl of a two-dimensional Gaussian free field, the same approach was adopted in \cite{CMOW22,ABK24}, and through a ``progressive'' homogenization, the precise order of the mean square displacement and a quenched central limit theorem were derived. Similar results were obtained in \cite{CLF22,FH22}, taking inspirations from \cite{LQSY04,Yau04}.

\subsection{Sketch of the argument and main challenges}
\label{s.sketch}
In this section, we will formulate the problem as a diffusion in a random environment and explain on a formal level the main steps of the proof. Some proofs will be written in different ways later (as always), but we are hoping that a bird's-eye view at this stage will help the readers to understand the ideas in the argument.

To study the polymer endpoint, through a Girsanov transformation, it is equivalent with analyzing the following diffusion with a random drift:
\begin{equation}\label{e.sde1Intro}
d\cX_s=u(t-s,\cX_s)ds+dB_s,  \quad\quad \cX_0=0,
\end{equation}
where $B$ is a standard Brownian motion and $u$ is the stationary solution to the stochastic Burgers equation
\begin{equation}\label{e.burgersIntro}
\partial_t u=\tfrac12\Delta u+\tfrac12\nabla u^2+\beta\nabla \xi.
\end{equation}
Here the driving force $\xi$ of the evolving random environment is independent of the thermal noise $B$.
Then $\cX_t$ can be written as 
\[
\cX_t=\int_0^t u(t-s,\cX_s)ds+B_t,
\]
and, as in standard homogenization, the first step is to extract the
martingale part from the drift, see e.g. \cite[Chapters 9 and
  11]{KLO12}. To do that, we consider the following equation 
\begin{equation}\label{e.eqphiIntro}
\partial_t \phi=\frac12\Delta\phi+u\nabla \phi+u,\quad \phi(0,x)= 0,
\end{equation}
and call $\phi$ the corrector, which could be obtained through a formal two-scale expansion of the backward Kolmogorov equation with drift $u$. \blue{More precisely, consider the equation 
\[
\partial_t f_\eps=\tfrac12\Delta f_\eps+\tfrac{1}{\eps}u(\tfrac{t}{\eps^2},\tfrac{x}{\eps})\nabla f_\eps,
\]
if we plug in the ansatz $f_\eps(t,x)=\bar f(t,x)+\eps {\frak
    f}_1(t,x,s,y)+\eps^2 {\frak f}_2(t,x,s,y)+\ldots$, with the fast
variables $s=\tfrac{t}{\eps^2},y=\tfrac{x}{\eps}$, and identity the terms of the same
order in $\eps$, then we obtain ${\frak f}_1(t,x,s,y)=\partial_x\bar f(t,x)
\phi(s,y)$,  where $\phi$ solves an equation of the form
\eqref{e.eqphiIntro}.}


There are different ways of formulating the corrector equation which may or may not incorporate the evolution of the underlying environment. For our purpose, it is more convenient to consider the above one, where the dynamics of $u$ is not taken into account. One can show that (with $\E_B$ denoting the expectation only on $B$)
\begin{equation}\label{e.phiIntro}
\phi(t,0)=\E_B\int_0^t u(t-s,\cX_s)ds, 
\end{equation}
and 
\begin{equation}\label{e.madeIntro}
\cX_t=\phi(t,0)+\int_0^t (1+\nabla\phi(t-s,\cX_s))dB_s.
\end{equation}
In this way,
one can view the corrector $\phi(t,0)$ as the quenched mean:
$\phi(t,0)=\E_B \cX_t$. Since $u$ is periodic in space, one can also
write $\phi(t,0)=\E_B \int_0^t u(t-s,\dot{\cX}_s)ds$, with $\dot{\cX}$
representing the fractional part of $\cX$. It is not hard to imagine
that both the solution to the Burgers equation and the process
$\dot{\cX}_s$ are fast mixing in time, so $\phi(t,0)$ can be viewed as
a sum of weakly dependent random variables with $\phi(t,0)/\sqrt{t}$
satisfying a central limit theorem. Therefore, both terms on the
r.h.s. of \eqref{e.madeIntro} contribute to the limiting effective
diffusivity. This is in sharp contrast to the case of a divergence form operator or a divergence free drift, where the corrector grows sublinearly hence does not contribute to the limiting diffusivity.


It turns out the martingale part $\int_0^t(1+\nabla\phi(t-s,\cX_s))dB_s$ is easy to study, because, for directed polymer in a random environment that is statistically shear-invariant (which is our case), through the Gibbs measure formulation, one can show that the mean of the quenched variance equals to $t$:
\[
\EE \E_B\int_0^t (1+\nabla\phi(t-s,\cX_s))^2ds=t.
\]
This is where the ``$1$'' factor in \eqref{e.exsigma} comes from. 

The main difficulty of the paper then reduces to deriving an explicit
variance of $\tfrac{\phi(t,0)}{\sqrt{t}}$. At this point, it is worth
emphasizing that $\cX_t$ does \emph{not} satisfy a quenched central
limit theorem. \blue{More precisely, for the polymer endpoint
  $\mathcal{X}_t$, one can view the $\phi(t,0)$ on the r.h.s. of
  \eqref{e.madeIntro} as describing the fluctuations resulting from the
  disorders, and the It\^o integral with respect to $B$ as the thermal
  fluctuations. In our periodic setting, we expect a quenched central limit theorem to hold for
  the It\^o integral term, while the $\phi(t,0)$ term is
  ``deterministic'' if we quench the random environment, and
  $\tfrac{\phi(t,0)}{\sqrt{t}}$ does not vanish as $t\gg1$.  A similar case was discussed in \cite{BF22}, with a non-quenched central limit theorem proved for a diffusion in a random environment. On a
different note, one should view the $\phi(t,0)$ term as a measurement of how ``localized'' the polymer endpoint is: suppose we sample two different polymer paths from the quenched Gibbs measure, they ``share'' the $\phi(t,0)$ term. To compare to the non-periodic setting, where $\mathcal{X}_t$ is expected to be of order $t^{2/3}$ and to localize around some favorite point, we suspect that, as the size of the
torus goes to infinity, this favorite
point comes from the $\phi(t,0)$ term, which is a functional of the random environment. Meanwhile, the It\^o
  integral term becomes of order $O(1)$ as $t\to\infty$, so $\mathcal{X}_t$ stays $O(1)$ distance away from $\phi(t,0)$. The second moment of the It\^o integral term blows up as $t\to\infty$, but that is due to the heavy tail \cite{DGL23}}.


To study the quenched mean, it is more convenient to consider the Gibbs formulation. We study $X_t=\int_{\R} x\rho(t,x)dx$, which has the same law as $\phi(t,0)$.  Applying the Clark-Ocone formula, we can write $X_t$ as an It\^o integral, with respect to the underlying white noise $\xi$:
\begin{equation}\label{e.coIntro}
X_t=\int_0^t\int_{\bT} f(s,y) \xi(s,y)dyds,
\end{equation}
where $f(s,y)=\EE[\D_{s,y}X_t|\F_s]$ with $\D_{s,y}$ denoting
the Malliavin derivative operator, and $\F_s$ is the natural
filtration generated by $\xi$. The calculation of $\D_{s,y}X_t$
is rather technical, but in the end we obtain $f(s,\cdot)$ as a
functional of $\big(u(s,\cdot),\nabla\phi(s,\cdot)\big)$ (which is a surprise to us). Thus, to
derive the formula for the variance of $X_t$, one needs to understand the joint distribution of $u(s,\cdot)$ and $\nabla\phi(s,0)$. 

Define $g=1+\nabla\phi$,  we know  from \eqref{e.eqphiIntro} that $g$ satisfies the Fokker-Planck equation 
\begin{equation}
\partial_tg=\frac12\Delta g+\nabla (ug).
\end{equation}
In this way, the problem reduces to understanding the joint distribution of $u(s,\cdot)$ and $g(s,\cdot)$. One should note that, for general drifts, this is not possible. Here is another key point in the argument: $g$ is the density of another diffusion in random environment:
\begin{equation}\label{e.sde2Intro}
d\cY_s=-u(s,\cY_s)ds+dB_s.
\end{equation}
Compare the two SDEs satisfied by $\cX$ and $\cY$, we note that the
random drifts are time reversal of each other, up to the opposite
sign. Roughly speaking, since $\cY$  follows approximately the
characteristics of the Burgers flow, one may suspect that it corresponds to the
so-called second class particle in the context of exclusion processes. For the drift given by \eqref{e.burgersIntro},
there is a  time reversal anti-symmetry
\begin{equation}\label{e.antisymmetry}
\{u(t-s,y)\}_{s\in[0,t],y\in\bT}\stackrel{\text{law}}{=}\{-u(s,y)\}_{s\in[0,t],y\in\bT}
\end{equation}
that holds for any $t>0$.
Using the above equality in law, one can re-connect $\cY$ to $\cX$ and derive the joint distribution of $u(s,\cdot)$ and $g(s,\cdot)$, through the Gibbs formulation of the directed polymers. This, together with the Clark-Ocone representation \eqref{e.coIntro}, eventually helps us to compute the asymptotic variance of $X_t/\sqrt{t}$, which corresponds to the second term on the r.h.s. of \eqref{e.exsigma}.

\blue{Overall, the fact that $\sigma^2(\beta)$ can be written
  explicitly in terms of the underlying invariant measure comes
    as a surprise to us, and the situation here seems to be more
  subtle than the study of  the free energy $\log \bar{Z}_t$.  For a
  noise that is white in time but smooth in space, we have obtained a
  formula for the variance of $\log \bar{Z}_t$, similar to the one for
  $\Sigma^2(\beta)$ given in \eqref{e.gammaSigma}, see
  \cite[Eq. (5.6)]{ADYGTK22}. For the variance of the polymer
  endpoint, our argument relies crucially on the time reversal
  anti-symmetry in \eqref{e.antisymmetry}, which does not hold for a
  colored noise, and this is actually the only place in the paper
  where   the noise is required to be white. It is also unclear to us
  whether there exists a similar result  for a colored noise. On a heuristic level, the time reversal anti-symmetry allows us to go back and forth between the Gibbsian and the Markovian settings, connecting the polymer measure, which is the diffusion with a non-Markovian drift described by \eqref{e.sde1Intro}, to the ``second class particle'' diffusion with a Markovian drift, given in \eqref{e.sde2Intro}. The zero-temperature version of the polymer model is the last passage percolation, and one may want to draw the connection between the two diffusions and the duality between the geodesics and the competing interfaces \cite{Sep20}.  In the end, the Gibbs measure formulation allows us to express the joint distribution of the endpoint density ``$g$'' and the drift ``$u$'' explicitly. 
}

\subsection{Organization of the paper.}
From the above discussion, we immediately identify a few difficulties
in implementing the argument. (i) To study the quenched mean
$\phi(t,0)$, it is actually more convenient to use the Gibbs measure
formulation, where one can write the quenched mean as $\int_{\R}
x\rho(t,x)dx$, with the endpoint density $\rho$ defined in
\eqref{e.defrhoIntro}. To apply the Clark-Ocone formula, we need to
compute $\D_{s,y}\rho(t,x)$, which involves cancellations and is
considerably more complicated than computing the Malliavin derivative
of the free energy. This is the main task of Section~\ref{s.co}. (ii)
To use the time reversal anti-symmetry, we need to initiate the
Burgers equation \eqref{e.burgersIntro} at stationarity, but this is
not the case for the point-to-line directed polymer. Therefore, in
Section~\ref{s.winding}, we will derive an error estimate which will
help us to approximate the free boundary condition by the stationary
boundary condition. (iii) The SDEs \eqref{e.sde1Intro} and
\eqref{e.sde2Intro} are only symbolic  since the drift $u$ is
distribution-valued, and this is where the recently developed theories
for  singular diffusions play a role. In Section~\ref{s.scalar}, we will study the two related passive scalar problems and make use of the time reversal anti-symmetry to derive the joint distribution of $u(s,\cdot)$ and $\nabla\phi(s,\cdot)$. (iv) In Section~\ref{s.var}, we compute the variance of $\phi(t,0)$, using the fast mixing of $u(s,\cdot)$, or equivalently, the polymer endpoint density on the torus, and  complete the proof of Theorem~\ref{t.mainth}.

\subsection{Notations}

(i) We write $q_t(x)=\frac{1}{\sqrt{2\pi t}}e^{-\frac{x^2}{2t}}$ as the standard heat kernel on $\R$, and $G_t(x)=\sum_{n\in\Z} q_t(x+n)$ is the heat kernel on $\bT$. 

(ii) We   use $W_i$, with the index $i$, to represent independent Brownian bridges connecting $(0,0)$ and $(1,0)$, periodically extended to $\R$. We will use $B$ to represent standard Brownian motion on $\R$. 

(iii) Since there are different sources of randomness, we use $\EE$ to denote the expectation on $\xi$, and $\E_{W_i},\E_B$ to denote the expectations on $W_i,B$ respectively.

\subsection*{Acknowledgement} We thank \'Eric Brunet for the
discussion and encouragement, and thank  the anonymous referees for  very helpful suggestions and comments. Y. G. was partially supported by the NSF
through DMS-2203014.  T. K. acknowledges the support of NCN grant 2020/37/B/ST1/00426.

\section{Quenched mean and Clark-Ocone representation}
\label{s.co}

The goal of this section is to present the first key step in the proof
of Theorem~\ref{t.mainth}. The total variance of the polymer endpoint
is decomposed into two parts: the variance of the quenched mean and
the mean of the quenched variance. It turns out that for our model
with the noise satisfying the shear-invariance property,
    the mean of the quenched variance is easy to deal with, see \eqref{e.exqvar} below, which is sometimes regarded as a physics folklore, see e.g. \cite[Eq. (42)]{Bru03}. The key difficulty is then to study the fluctuations of the quenched mean. Our idea is similar to that of \cite{ADYGTK22}, where tools from Malliavin calculus were borrowed, and the quenched mean will be written as an explicit It\^o integral through the Clark-Ocone formula, see \eqref{e.coXt} below.

\subsection{Variance decomposition and quenched mean}

It was known from \cite[Eq. (3.21), Proposition 5.1]{YGTK22} that the effective diffusivity in \eqref{e.diffusivityendpoint} can be obtained through 
\begin{equation}\label{e.varcon}
\sigma^2(\beta)=\lim_{t\to\infty}\frac{1}{t}V_t, 
\end{equation}
where
\begin{equation}\label{e.varan}
  V_t:=\EE\int_{\R} x^2\rho(t,x)dx
\end{equation}
is the annealed variance.
  Let
\begin{align}
  \label{e.defXt}
    X_t:=\int_{\R} x\rho(t,x)dx,\qquad
\V_t:=\int_{\R} x^2\rho(t,x)dx-X_t^2
\end{align}
be the quenched mean and the quenched variance of the polymer endpoint, respectively.
Then, we can obviously write
\begin{equation}\label{e.devar}
V_t=\EE X_t^2 + \EE \V_t.
\end{equation}
It is folklore, e.g. see a proof in \cite[Lemma 5.3]{YGTK22}, that the average of the
quenched variance is $t$, i.e.,
\begin{equation}\label{e.exqvar}
\EE \V_t=t.
\end{equation}
Thus, to derive an expression for $\sigma^2(\beta)$, it suffices to study the first term on the r.h.s. of \eqref{e.devar}, which is actually the variance of the quenched mean. 

Let us introduce some notations. Let $\cZ_{t,s}(x,y)$ be the propagator of the SHE and $\G_{t,s}(x,y)$ be its periodic counterpart. Namely, for any fixed $(s,y)$, $\cZ_{t,s}(x,y)$ solves 
\begin{equation}\label{e.defGreen}
\begin{aligned}
&\partial_t \cZ_{t,s}(x,y)=\frac12\Delta_x \cZ_{t,s}(x,y)+\beta \cZ_{t,s}(x,y)\xi(t,x), \quad\quad t>s,x\in\R,\\
&\cZ_{s,s}(x,y)=\delta(x-y), 
\end{aligned}
\end{equation}
and
\[
\G_{t,s}(x,y)=\sum_{n \in\Z} \cZ_{t,s}(x,y+n)=\sum_{n \in\Z} \cZ_{t,s}(x+n,y).
\]
So we know that the solution to \eqref{e.she} is actually $Z(t,x)=\cZ_{t,0}(x,0)$. 

It is worth emphasizing that one should view $\cZ $ as the propagator
of SHE on $\R_+\times \R$ while $\G $ as the propagator of the same
equation on $\R_+\times \bT$. This leads to the fact that, if we view
the polymer path as lying on the   half plane $\R_+\times \R$, then the quenched density should be expressed in terms of $\cZ$; while if we view the polymer path as lying on the half cylinder $\R_+\times \bT$, the quenched density should be expressed in terms of $\G$.

Next, we introduce the following notations on the endpoint densities
of the ``forward'' and ``backward'' polymer path on a cylinder. 
Let ${\mathcal M}_1(\bT)$ be the set of all Borel probability
measures  on $\bT$. 
For
any $\nu\in{\mathcal M}_1(\bT)$ and $t>s$, we let 
\begin{equation}\label{e.forwardbackward}
\begin{aligned}
&\rhof(t,x;s,\nu)=\frac{\int_{\bT} \G_{t,s}(x,y)\nu(dy)}{\int_{\bT^2}\G_{t,s}(x',y)\nu(dy)dx'}, \quad x\in\bT,\\
&\rhob(t,\nu;s,y)=\frac{\int_{\bT}\G_{t,s}(x,y)\nu(dx)}{\int_{\bT^2}\G_{t,s}(x,y')\nu(dx)dy'},\quad y\in\bT.
\end{aligned}
\end{equation}
Here the subscripts ``$\mathrm{f}, \mathrm{b}$'' represent ``forward'' and ``backward'' respectively. When $\nu$ is a Dirac measure on $\bT$, we   write
$$\rhof(t,x;s,y)=\rhof(t,x;s,\delta_y)\quad \mbox{and}
\quad \rhob(t,x;s,y)=\rhob(t,\delta_x;s,y).
$$

We also need the simplified notation 
\begin{equation}\label{e.defbarG}
\G_{t,s}(-,y)=\int_{\bT} \G_{t,s}(x,y)dx=\int_{\R} \cZ_{t,s}(x,y)dx,
\end{equation}
which can be viewed as the point-to-line partition function with the starting point $(s,y)$.
%
%
%

For any $\theta\in\R$, define  
\begin{equation}\label{e.defUtheta}
\mathcal{E}_\theta(t)=\int_{\R} e^{\theta x}\cZ_{t,0}(x,0)dx,
\end{equation}
so one can write 
\begin{equation}\label{e.Xtde}
X_t=\partial_\theta \log \cE_\theta(t)\bigg|_{\theta=0}.
\end{equation}
It is clear that $X_t$ is a square integrable random variable that is
measurable with respect to $\F_t$ - the natural filtration
corresponding to the noise $\{\xi(s,y)\}_{(s,y)\in\R\times\bT}$. The idea is to apply
the Clark-Ocone formula to write $X_t$ as an It\^o integral with respect to $\xi$, through which we compute the second moment.

\subsection{Clark-Ocone representation}
We will first write $\log \cE_\theta(t)-\EE \log \cE_\theta(t)$ as an It\^o integral, then take derivative on $\theta$ to obtain the representation for $X_t$.

Recall that we view $\xi$ as the spacetime white noise on $\R\times \bT$. Let
$\D_{s,y}$ denote the Malliavin derivative \cite[Chapter 1]{Nua06}. Namely, for ``smooth'' random variables $X$, $(\D_{s,y}X)_{(s,y)\in\R\times \bT}$ is an $L^2(\R\times \bT)-$valued random variable.

By the Clark-Ocone formula, see e.g. \cite[Proposition 6.3]{CKNP19}, we have 
\begin{equation}\label{e.coU}
\log \cE_\theta(t)-\EE \log \cE_\theta(t)=\int_0^t\int_{\bT} \EE [\D_{s,y}\log \cE_\theta(t)|\F_s] \xi(s,y)dyds.
\end{equation}

First, let us consider the mean.  Fix $t>0$. Since
\begin{equation}
  \label{010905-23}
  \{\cZ_{t,0}(x,0)\}_{x\in\R}\stackrel{\rm
    law}{=}\{\cZ_{t,0}(0,x)\}_{x\in\R},
\end{equation}
we have
\[
\begin{aligned}
 & \EE \log \cE_\theta(t)=\EE \log\left( \int_{\R} e^{\theta x}\cZ_{t,0}(x,0)dx\right)=
\EE \log\left( \int_{\R} \cZ_{t,0}(0,x)e^{\theta x} dx\right)\\
&=\EE \log\left( \int_{\R} \cZ_{t,0}(0,x)  dx\right)+\frac{\theta^2 t}{2}.
\end{aligned}
\]
For the last ``='', see the proof of Lemma 5.3 in \cite{YGTK22}.
Thus, we conclude that
\begin{equation}
  \label{020805-23}
\EE \log \cE_\theta(t)=\EE \log\cE_0(t)+\frac12\theta^2 t.
\end{equation}
Here is the lemma on the Malliavin derivative:
\begin{lemma}\label{l.made}
For any $s<t$ and $y\in \bT$, we have 
\begin{equation}\label{e.madere}
\D_{s,y}\cE_\theta(t)=\beta\sum_{n\in\Z} \cZ_{s,0}(y+n,0)\int_{\R}e^{\theta x}\cZ_{t,s}(x,y+n)dx.
\end{equation}
\end{lemma}
 The proof of the lemma  is  standard and follows from an approximation
argument and application of  the Feynman-Kac formula (see
\cite{BC95}), so we only sketch it here. To consider some
approximation, let us introduce a few notations that will be used
throughout the paper. For a given $\eps>0$, let $\xi^\eps$ be a spatial mollification of $\xi$: 
\[
\xi^\eps(t,x)=\int_{\R}\eta^\eps(x-y)\xi(t,y)dy, \quad \quad t,x\in\R
\] where $\eta^\eps$ is a symmetric  approximation of the Dirac
  function, as $\eps\to0$, that is supported in $[-1/4,1/4]$.  Let $R^\eps$ be the
spatial covariance function of $\xi^\eps$, which is the periodic extension of $\eta^\eps\star \eta^\eps$:
\[
R^\eps(x)=\sum_{n\in\Z} \eta^\eps\star \eta^\eps(x+n),
\] and let $\cZ^\eps$ be the propagator of the SHE with the mollified
noise $\xi^\eps$, defined in the same way as
\eqref{e.defGreen}. Finally, let $\cE_\theta^\eps(t)=\int_{\R} e^{\theta x} \cZ^\eps_{t,0}(x,0)dx$.

\begin{proof}[Proof of Lemma~\ref{l.made}]
By the Feynman-Kac formula, see \cite[Theorem 2.2]{BC95}, we can write 
\[
\cE_\theta^\eps(t)=\E_B \left[e^{\beta \int_0^t \xi^\eps(\ell,B_\ell)d\ell-\frac12\beta^2R^\eps(0)t} e^{\theta B_t}\right],
\]
where $B$ is a standard Brownian motion starting from the origin and the
expectation $\E_B$ is taken only with respect to $B$. Since
$\cE_\theta^\eps(t)\to \cE_\theta(t)$ in $L^2(\Omega)$, it suffices to
show that the convergence of $\beta^{-1}\D_{s,y}\cE_\theta^\eps(t)$ to
the r.h.s. of \eqref{e.madere} holds in $L^2(\Omega,L^2([0,t]\times \bT))$.

By the definition of $\xi^\eps$, we have
\[
\int_0^t \xi^\eps(\ell,B_\ell)d\ell=\int_0^t\int_{\bT} \sum_{n\in\Z} \eta^\eps(B_\ell-y-n)\xi(\ell,y)dyd\ell,
\]
so for each fixed realization $B$, it holds that 
\[
\D_{s,y}\left(\int_0^t \xi^\eps(\ell,B_\ell)d\ell\right)=\sum_{n\in\Z} \eta^\eps(B_s-y-n).
\]
This leads to 
\[
\begin{aligned}
&\beta^{-1}\D_{s,y}\cE_\theta^\eps(t)=\sum_{n\in\Z}\E_B \left[e^{\beta \int_0^t \xi^\eps(\ell,B_\ell)d\ell-\frac12\beta^2R^\eps(0)t}\eta^\eps(B_s-y-n)e^{\theta B_t}\right]\\
&=\sum_{n\in\Z}\int_{\R} q_s(w)\eta^\eps(w-y-n)\E_B[e^{\beta\int_0^t \xi^\eps(\ell,B_\ell)d\ell-\frac12\beta^2R^\eps(0)t} e^{\theta B_t}|B_s=w]dw\\
&=\sum_{n\in\Z}\int_\R \eta^\eps(w-y-n)\left(\int_{\R}e^{\theta x} \cZ^\eps_{t,s}(x,w) dx\right)\cZ^\eps_{s,0}(w,0)dw.
\end{aligned}
\]
Sending $\eps\to0$, we complete the proof.
\end{proof}


 Now we can state the following proposition: 
\begin{proposition}\label{p.made}
For any $0\le s<t$ and $y\in\bT$, we have 
\[
\begin{aligned}
& \EE [\D_{s,y}\log \cE_\theta(t)|\F_s]\\
&=\beta\EE \left[\frac{ \G_{t,s}(-,y)\sum_{n\in\Z}  e^{\theta (y+n)}\cZ_{s,0}(y+n,0)}{\int_{\R} \G_{t,s}(-,y')e^{\theta  y'}\cZ_{s,0}(y',0)dy'}\bigg|\F_s\right].
\end{aligned}
\]
\end{proposition}

\begin{proof}
We proceed by approximation. Following the proof of
Lemma~\ref{l.made}, we have 
\[
\frac{1}{\beta}\frac{\D_{s,y}\cE_\theta^\eps(t)}{\cE_\theta^\eps(t)}=\frac{\sum_{n\in\Z}\E_B \left[e^{\beta\int_0^t \xi^\eps(\ell,B_\ell)d\ell-\frac12\beta^2R^\eps(0)t}\eta^\eps(B_s-y-n)e^{\theta B_t}\right]}{\E_B \left[e^{\beta \int_0^t \xi^\eps(\ell,B_\ell)d\ell-\frac12\beta^2R^\eps(0)t} e^{\theta B_t}\right]}.
\]
The main challenge is hence to compute the conditional expectation given $\F_s$. To continue, first, for both the numerator and the denominator, we write 
\[
e^{\theta B_t}=e^{\theta B_s}e^{\frac12\theta^2(t-s)}\cdot e^{\theta(B_t-B_s)-\frac12\theta^2(t-s)},
\]
 and view the last exponential factor on the right as a change of measure. Applying
 the Girsanov theorem to the Brownian motion in $[s,t]$, we have
\begin{equation}\label{e.madeeps}
\frac{1}{\beta}\frac{\D_{s,y}\cE_\theta^\eps(t)}{\cE_\theta^\eps(t)}=\frac{\sum_{n\in\Z}\E_B \left[e^{\beta \int_0^t \xi^\eps(\ell,\tilde{B}^\theta_\ell)d\ell-\frac12\beta^2R^\eps(0)t}\eta^\eps(B_s-y-n)e^{\theta B_s}\right]}{\E_B \left[e^{\beta \int_0^t \xi^\eps(\ell,\tilde{B}^\theta_\ell)d\ell-\frac12\beta^2R^\eps(0)t} e^{\theta B_s}\right]},
\end{equation}
where $\{\tilde{B}^\theta_\ell\}_{\ell}$ 
  is obtained from $\{B_\ell\}_\ell$ by adding the drift $\theta$ in the interval $[s,t]$:
\[
\tilde{B}^\theta_\ell=B_\ell 1_{\ell\in[0,s]}+[B_\ell+\theta(\ell-s)]1_{\ell\in(s,t]}.
\]
In this way, we can rewrite 
\begin{equation}\label{e.4142}
\int_0^t \xi^\eps(\ell,\tilde{B}^\theta_\ell)d\ell=\int_0^s \xi^\eps(\ell,B_\ell)d\ell+\int_0^{t-s}\xi^\eps(s+\ell,B_{s+\ell}+\theta \ell )d\ell.
\end{equation}
Here is the key of the argument: to compute $\EE
[\tfrac{D_{s,y}\cE_\theta^\eps(t)}{\cE_\theta^\eps(t)}|\F_s]$, we need
to average out the noise in the domain $[s,t]\times \R$, but by \eqref{e.madeeps} and  \eqref{e.4142}, we see that the dependence on the noise beyond $s$ is only through the factor $\int_0^{t-s}\xi^\eps(s+\ell,B_{s+\ell}+\theta \ell )d\ell
$. Since $\xi^\eps$ is white in time and stationary in space, we have, for fixed $s$, 
\[
\{\xi^\eps(s+\ell,x+\theta \ell)\}_{\ell\geq0, x\in\R}\stackrel{\rm law}{=}\{\xi^\eps(s+\ell,x)\}_{\ell\geq0, x\in\R}.
\]
In other words, when computing the conditional expectation, one can
replace $\int_0^t \xi^\eps(\ell,\tilde{B}^\theta_\ell)d\ell$ by
\[
 \int_0^s \xi^\eps(\ell,B_\ell)d\ell+\int_0^{t-s}\xi^\eps(s+\ell,B_{s+\ell} )d\ell=\int_0^t \xi^\eps(\ell,B_\ell)d\ell.
\]
As a result
\begin{equation}
  \label{010805-23}
    \begin{aligned}
&\frac{1}{\beta}\EE\left[\frac{\D_{s,y}\cE_\theta^\eps(t)}{\cE_\theta^\eps(t)}\bigg|\F_s\right]\\
&=\EE \left[\frac{\sum_{n\in\Z}\E_B \left[e^{\beta\int_0^t \xi^\eps(\ell,B_\ell)d\ell-\frac12\beta^2R^\eps(0)t}\eta^\eps(B_s-y-n)e^{\theta B_s}\right]}{\E_B \left[e^{\beta \int_0^t \xi^\eps(\ell,B_\ell)d\ell-\frac12\beta^2R^\eps(0)t} e^{\theta B_s}\right]}\bigg|\F_s\right].
\end{aligned}
\end{equation}
Now we can pass to the limit, as $\eps\to0$, and arguing as in the
proof of Lemma~\ref{l.made}, we conclude that the right hand side of
\eqref{010805-23} tends to
\[
\EE \left[\frac{\sum_{n\in\Z} \G_{t,s}(-,y+n)\cdot  e^{\theta (y+n)}\cZ_{s,0}(y+n,0)}{\int_{\R} \G_{t,s}(-,y')e^{\theta  y'}\cZ_{s,0}(y',0)dy'}\bigg|\F_s\right],
\]
where we recall that $\G_{t,s}(-,y)=\int_{\R} \cZ_{t,s}(x,y)dx$. It is easy to see that $y\mapsto\G_{t,s}(-,y)$ is $1$-periodic in $y$, so the proof is complete.
\end{proof}

By \eqref{e.Xtde}, computing derivative on $\theta$ on both sides of
\eqref{e.coU} and using \eqref{020805-23}, we obtain
\[
X_t=\frac{d}{d\theta}_{|\theta=0}\log
\cE_\theta(t)=\frac{d}{d\theta}_{|\theta=0}\int_0^t\int_{\bT} \EE
[\D_{s,y}\log \cE_\theta(t)|\F_s] \xi(s,y)dyds .
\] 
Applying Proposition~\ref{p.made}, we further derive
\begin{equation}\label{e.coXt}
X_t=\beta \int_0^t\int_{\bT} \EE[I_{1,t}(s,y)-I_{2,t}(s,y)|\F_s] \xi(s,y)dyds,
\end{equation}
with 
\begin{equation}\label{e.defi1}
\begin{aligned}
&I_{1,t}(s,y)=\frac{ \G_{t,s}(-,y) \sum_{n\in\Z} (y+n)\cZ_{s,0}(y+n,0)}{\G_{t,0}(-,0)},
\end{aligned}
\end{equation}
and
\begin{equation}\label{e.defi2}
\begin{aligned}
I_{2,t}(s,y)=\rho_{\rm m}(t,-|s,y)\frac{\int_{\R} \G_{t,s}(-,y')y'\cZ_{s,0}(y',0)dy'}{\G_{t,0}(-,0)},
\end{aligned}
\end{equation}
where
%
\begin{equation}\label{e.defrho}
\rho_{\rm m}(t,-|s,y)=\frac{ \G_{t,s}(-,y) \G_{s,0}(y,0)}{\G_{t,0}(-,0)},
\end{equation}
is  the midpoint density  of the polymer, i.e. it is the quenched density at time $s$ of the point-to-line polymer, starting at $(0,0)$. To simplify the notation,  we omitted the dependence on the starting point in the expression of $\rho_{\rm m}$. In particular, we have $\int_{\bT}
\rho_{\rm m}(t,-|s,y)dy=1$. Here the subscript ``${\rm m}$'' represents ``middle''.

For any $s>0, y\in\bT$, define 
\begin{equation}\label{e.defW}
\psi(s,y)=-\frac{\sum_{n\in\Z} (y+n)\cZ_{s,0}(y+n,0)}{\G_{s,0}(y,0)}=\frac{\sum_{n\in\Z} (n-y)\cZ_{s,0}(y,n)}{\G_{s,0}(y,0)}.
\end{equation}
\blue{Here in the second ``='', we changed $n\to -n$ and used the fact that $\cZ_{s,0}(y-n,0)=\cZ_{s,0}(y,n)$.}
 One can interpret $\psi(s,y)$ as the ``winding number'' of the
point-to-point polymer path on the cylinder, 
  starting at $(s,y)$ and ending at $(0,0)$, see more discussion in Section~\ref{s.winding}.
The last equality follows from the fact that in the case of
a $1$-periodic noise, we have
$$
\cZ_{s,0}(y-n,0)=\cZ_{s,0}(y,n),\quad n\in\Z.
$$

We have the following lemma, which expresses $I_{1,t}-I_{2,t}$ in
terms of $\rho_{\rm m}$ and $\psi$ and will serve \blue{as} the starting point of the rest of the analysis.
\begin{lemma}\label{l.i12}
We have 
\begin{equation}\label{e.i1minusi2}
I_{1,t}(s,y)-I_{2,t}(s,y)=\rho_{\rm m}(t,-|s,y)\int_{\bT}[\psi(s,y')-\psi(s,y)]\rho_{\rm m}(t,-|s,y')dy'.
\end{equation}
\end{lemma}
\begin{proof}
First, using the definition \eqref{e.defrho}, we can rewrite \eqref{e.defi1} in the form
\[
\begin{aligned}
I_{1,t}(s,y)&= -\rho_{\rm m}(t,-|s,y)\psi(s,y).
\end{aligned}
\]
Then, we   rewrite the numerator of the second factor on the
right hand side of \eqref{e.defi2} as
\[
\begin{aligned}
&\int_{\R} \G_{t,s}(-,y')y'\cZ_{s,0}(y',0)dy'\\
&=\int_{\bT} \G_{t,s}(-,y')\G_{s,0}(y',0)\frac{\sum_{n\in\Z} (y'+n)\cZ_{s,0}(y'+n,0)}{\G_{s,0}(y',0)}dy',
\end{aligned}
\]
where we used the fact that $\G_{t,s}(-,y)$ is $1-$periodic in $y$.
Thus, invoking the definitions \eqref{e.defrho} and 
\eqref{e.defW},  we can rewrite
\eqref{e.defi2} in the form
\[
\begin{aligned}
I_{2,t}(s,y)&=\rho_{\rm m}(t,-|s,y)\int_{\bT}\rho_{\rm m}(t,-|s,y') \frac{\sum_{n\in\Z} (y'+n)\cZ_{s,0}(y'+n,0)}{\G_{s,0}(y',0)}dy'\\
&=-\rho_{\rm m}(t,-|s,y)\int_{\bT}\rho_{\rm m}(t,-|s,y')\psi(s,y')dy'.
\end{aligned}
\] 
The proof is complete.
\end{proof}

To summarize, the main results of this section are \eqref{e.coXt} and
\eqref{e.i1minusi2}. They allow  to express the quenched mean
$X_t=\int_{\R} x\rho(t,x)dx$ as a more or less explicit It\^o
integral. The integrand involves the quenched midpoint density
$\rho_{\rm m}(t,-|s,y)$ (see \eqref{e.i1minusi2}), which we understand
well, and the difference between ``winding numbers'' of the polymer
paths starting at $(s,y')$ and $(s,y)$, i.e. the factor
$\psi(s,y')-\psi(s,y)$, which we do not understand so well  so far
and presents itself the main challenge. In a sense, to derive an explicit variance formula for $X_t$, one needs to understand the joint distribution of $\rho_{\rm m}(t,-|s,y)$ and $\psi(s,y')-\psi(s,y)$. This will be the content of the next two sections.

\section{Winding number and stabilization}
\label{s.winding}

In this section, let us look more closely at the function $\psi(s,y)$
defined in \eqref{e.defW}.
Recall that a key quantity in the Clark-Ocone representation
\eqref{e.coXt}  is $\psi(s,y')-\psi(s,y)$. 
The goal of this section is to reframe the problem so that we are in a  stationary setting, and the main result, Proposition~\ref{p.exerr} below,  proves a stationary approximation of $\psi(s,y')-\psi(s,y)$.

To better understand it, we define
an $\eps-$approximation of $\psi(s,y)$. Recall that $\cZ^\eps$ is the propagator of
SHE with noise $\xi^\eps$. Let also $\G^\eps$ be the $\eps-$counterpart of $\G$, and define 
\[
\psi^\eps(s,y) 
=\frac{\sum_{n\in\Z} (n-y)\cZ^\eps_{s,0}(y,n)}{\G^\eps_{s,0}(y,0)}.
\]
By the Feynman-Kac formula, we can rewrite $\psi^\eps$ and interpret its
integer part (floor, or ceiling depending on whether it is positive or
negative \blue{respectively}) as
the winding number of a directed polymer on a cylinder:

\begin{lemma}
For any $s>0$ and $y\in \bT$, we have 
\begin{equation}\label{e.fkpsieps}
\psi^\eps(s,y)= \frac{\E_B \left[ e^{\beta\int_0^s\xi^\eps(s-\ell,B_\ell)d\ell-\frac12\beta^2R^\eps(0)s} ( B_s-y)\,|\,  B_0=y,\dot{B}_s=0\right]}{\E_B \left[ e^{\beta\int_0^s\xi^\eps(s-\ell,B_\ell)d\ell-\frac12\beta^2R^\eps(0)s}  \,|\, B_0=y,\dot{B}_s=0\right]},
\end{equation}
where $\dot{B}$ denotes the fractional part of $B$ (in other words, $\dot{B}$ is a Brownian
  motion on $\bT$).
\end{lemma}

\begin{proof}
To verify the above formula, we rewrite the numerator of the
expression on the right hand side of \eqref{e.fkpsieps}:
\[
\begin{aligned}
&\E_B \left[ e^{\beta\int_0^s\xi^\eps(s-\ell,B_\ell)d\ell-\frac12\beta^2R^\eps(0)s} (B_s-y)\,|\, B_0=y,\dot{B}_s=0\right]\\
&=\frac{\sum_{n\in\Z} \E_B \left[ e^{\beta\int_0^s\xi^\eps(s-\ell,B_\ell)d\ell-\frac12\beta^2R^\eps(0)s} (B_s-y)\delta(B_s-n)|B_0=y\right]}{\sum_{n\in\Z}q_s(n-y)}\\
&=\frac{\sum_{n\in\Z} (n-y)\cZ^\eps_{s,0}(y,n)}{G_s (y)}.
\end{aligned}
\]
The
denominator can be  rewritten in the same way. \end{proof}

The above lemma interprets $\psi^\eps(s,y)$ as the winding number, up
to a fractional part, of the polymer path around the cylinder, with
the starting point $(s,y)$ and the ending point $(0,0)$. Due to the
fast mixing  of the polymer endpoint distribution on the torus,  the
prescribed distribution of the ending point plays no role in the long time, and it is actually more convenient to consider the ending point at stationarity. 
For this purpose,
we define the key object 
\begin{equation}\label{e.defphi}
\phi(s,y)=\frac{\int_{\R} \cZ_{s,0}(y,y')(y'-y)e^{\beta W(y')}dy'}{\int_{\R} \cZ_{s,0}(y,y')e^{\beta W(y')}dy'},
\end{equation}
where $W$ is a Brownian bridge on $\bT$, periodically extended
to $\R$.  One should view $\phi$ as a ``stationary'' approximation of $\psi$, and the main result of the section is the following estimate,
which will help us to eventually  replace the $\psi$ factors in
\eqref{e.i1minusi2} by the corresponding $\phi$ factors. 

\begin{proposition}\label{p.exerr}
 For any $p\in[1,\infty)$, there exists $C,\lambda>0$ depending on $p$ so that 
 \begin{equation}\label{e.exerr}
 \sup_{y,y'\in\bT}\E_W\EE
 |[\psi(s,y')-\psi(s,y)]-[\phi(s,y')-\phi(s,y)]|^p \leq C e^{-\lambda
   s},\quad  s\ge0.
 \end{equation}
 \end{proposition}

Before commenting on the proof of Proposition \ref{p.exerr}, let us introduce some notation: for two probability  measures $\mu,\nu$ on $\bT$, we periodically extend $\nu$ to $\R$, and define 
\begin{equation}\label{e.defwinding}
  \begin{split}
&\cW(s;\mu,\nu)= \frac{\int_{\bT}\left(\int_{\R}
    \cZ_{s,0}(y,y')(y'-y)\nu(dy') \right) \mu(dy) }{\int_{\bT}
  \left(\int_{\R}\cZ_{s,0}(y,y')\nu(dy') \right) \mu(dy)}\\
&
=\frac{\int_{\bT}\left(\int_{\R}
    \cZ_{s,0}(y,y')(y'-y)\nu(dy') \right) \mu(dy) }{\int_{\bT}
  \left(\int_{\bT}\G_{s,0}(y,y')\nu(dy') \right) \mu(dy)}.
\end{split}
\end{equation}
With some abuse of terminology ($\cW$ needs not
be an integer), we shall refer to $\cW(s;\mu,\nu)$ as the winding
number of the polymer path with initial distribution $\mu$ (at time $s$) and ending
distribution $\nu$ (at time $0$).

Using this notation we can write 
\begin{equation}\label{e.psiphiW}
\psi(s,y)=\cW(s;\delta_y,\delta_0), \quad\quad \phi(s,y)=\cW(s;\delta_y, \rho_{W}),
\end{equation}
where (with some abuse of notation) we have used $\rho_{W}$ to denote the (random) probability measure on $\bT$ with density 
\[
\rho_{W}(x)=\frac{e^{\beta W(x)}}{\int_{\bT} e^{\beta W(x')}dx'}.
\]

To see on a heuristic level why Proposition~\ref{p.exerr} holds, we note that, for the same ending point distribution $\nu$ but different initial distributions $\mu_1,\mu_2$, the quantity $\cW(s,\mu_1,\nu)-\cW(s,\mu_2,\nu)$ 
is of order $O(1)$ for  $s\gg 1$, \blue{see Corollary~\ref{cor011205-23} below,} and it mostly depends on the random environment $\xi(\ell,\cdot)$ with $s-\ell\sim O(1)$, by the fast mixing of the polymer endpoint density on the
torus, which will become more clear later. Therefore, for  $s\gg1$, its dependence on $\nu$ is exponentially small in $s$, which explains \eqref{e.exerr}.


In light of \eqref{e.psiphiW}, Proposition~\ref{p.exerr} is a
conclusion of the following more general result:
for any $p\in[1,\infty)$, there exists $C,\lambda>0$ depending on $p$ so that
\begin{equation}\label{e.Wsta}
\begin{aligned}
&\sup_{\mu_i,\nu_i,i=1,2}\EE \big|[\cW(s;\mu_1,\nu_1)-\cW(s,\mu_2,\nu_1)]-[\cW(s;\mu_1,\nu_2)-\cW(s;\mu_2,\nu_2)]\big|^p \\
&\leq C e^{-\lambda s},\quad s\ge0,
\end{aligned}
\end{equation}
where the supremum is taken with respect to all $\mu_1,\mu_2,\nu_1,\nu_2\in\mathcal{M}_1(\bT)$.

The proof of \eqref{e.Wsta} is given in Section \ref{sec3.4}
below. Roughly speaking, we will write the winding number $\cW(s;\mu,\nu)$ as the running sum of a Markov chain, with the chain representing the winding number accumulated by the polymer path during each time interval of length $1$. The fast mixing of the chain drives the exponential decay in \eqref{e.Wsta}. Before presenting the argument, we need some auxiliary results
contained in the following sections.

\subsection{Rewriting the winding number in terms of a Markov chain.}
Recall that $\cW(s;\mu,\nu)$ was defined in \eqref{e.defwinding}.  For
notational convenience in this section, it is more convenient to view
$\mu$ as the ending distribution and $\nu$ as the starting
distribution. We extend $\mu$ periodically to $\R$. The following
lemma holds:
\begin{lemma}
We have
\begin{equation}\label{e.wsym}
\cW(s;\mu,\nu)=-\frac{\int_{\R} \left(\int_{\bT} \cZ_{s,0}(y,y')(y-y')\nu(dy')\right) \mu(dy)}{\int_{\R} \left(\int_{\bT} \cZ_{s,0}(y,y')\nu(dy')\right) \mu(dy)}.
\end{equation}
\end{lemma}

\begin{proof}
To see why \eqref{e.wsym} holds, consider the numerator of $\cW(s;\mu,\nu)$ in \eqref{e.defwinding} first. By the periodicity of $\nu$ and $\cZ$, we can write 
\[
\begin{aligned}
&\int_{\bT}\left(\int_{\R}  \cZ_{s,0}(y,y')(y'-y)\nu(dy') \right) \mu(dy)\\
&=\sum_{n\in\Z} \int_{\bT^2}\cZ_{s,0}(y,y'+n)(y'+n-y)\nu(dy')\mu(dy) \\
&=\sum_{n\in\Z} \int_{\bT^2}\cZ_{s,0}(y-n,y')(y'+n-y)\nu(dy')\mu(dy)\\
&=-\sum_{n\in\Z} \int_{\bT^2} \cZ_{s,0}(y+n,y')(y+n-y')\nu(dy')\mu(dy)\\
&=-\int_{\R} \left(\int_{\bT} \cZ_{s,0}(y,y')(y-y')\nu(dy')\right) \mu(dy),
\end{aligned}
\]
since $\mu$ is also periodically extended. The
denominator is treated in the same way, which completes the proof of
\eqref{e.wsym}.
\end{proof}


For  given $\mu,\nu\in {\mathcal M}_1(\bT)$ and each realization of the
noise, we introduce a Markov chain so that the winding number of the
polymer path is written as the running sum of the chain. This has also been used in \cite{YGTK22} to prove the central limit theorem of the winding number, and we repeat it here for the convenience of readers.

Fix any $s>0$ and denote $N=\lf s\rf$. Fix $y\in\R, y'\in\bT$. We first write it as $y=j_{N+1}+x_{N+1}$ for some $j_{N+1}\in \Z$ and $x_{N+1}\in[0,1)$. Since $\cZ$ is the propagator of SHE, we have 
\[
\begin{aligned}
\cZ_{s,0}(y,y')&=\int_{\R}\cZ_{s,N}(j_{N+1}+x_{N+1},x)\cZ_{N,0}(x,y')dx\\
&=\sum_{j_N\in\Z}\int_{\bT} \cZ_{s,N}(j_{N+1}+x_{N+1},j_N+x_N)\cZ_{N,0}(j_N+x_N,y')dx_N.
\end{aligned}
\]
Iterating the above relation, we obtain
\[
\begin{aligned}
\cZ_{s,0}(y,y')= \sum_{j_1,\ldots,j_N\in\Z}&\int_{\bT^{N}}\cZ_{s,N}(j_{N+1}+x_{N+1},j_N+x_N) \\
&\times \prod_{k=1}^N \cZ_{k,k-1}(j_k+x_k,j_{k-1}+x_{k-1})d\x_{1,N},
\end{aligned}
\]
where we used the simplified notation $d\x_{1,N}=dx_1\ldots dx_{N}$ and the convention 
$
j_0=0,x_0=y'.
$
 In other words, in the above integration, we have decomposed   $\R=\cup_{j\in\Z} [j,j+1)$, then integrate in each interval and sum them up. One should think of the variable $j_k+x_k$ as representing the location of the polymer path at time $k$, with $j_k$ the integer part and $x_k$ the fractional part.

Now we make use of the periodicity:
\begin{equation}\label{e.peZG}
\begin{aligned}
&\sum_{j_k\in\Z}\cZ_{k,k-1}(j_k+x_k,j_{k-1}+x_{k-1})=\sum_{j_k\in\Z} \cZ_{k,k-1}(j_k-j_{k-1}+x_k,x_{k-1})\\
&=\sum_{j_k\in\Z} \cZ_{k,k-1}(j_k+x_k,x_{k-1})=\G_{k,k-1}(x_k,x_{k-1}),
\end{aligned}
\end{equation}
where $\G$ is the periodic propagator of SHE. Then we can write 
\begin{equation}\label{e.7151}
\begin{aligned}
\cZ_{s,0}(y,y')=\int_{\bT^{N}}  &\left(\sum_{j_1,\ldots,j_{N}\in\Z} \frac{\cZ_{s,N}(j_{N+1}+x_{N+1},j_N+x_N)\prod_{k=1}^N \cZ_{k,k-1}(j_k+x_k,j_{k-1}+x_{k-1})  }{\G_{s,N}(x_{N+1},x_N)\prod_{k=1}^N \G_{k,k-1}(x_k,x_{k-1}) }\right)\\
&\times \G_{s,N}(x_{N+1},x_N)\prod_{k=1}^N \G_{k,k-1}(x_k,x_{k-1})  dx_{1,N}.
\end{aligned}
\end{equation}

Fix the realization of the random noise and  the vector
\begin{equation}\label{e.bx}
\x:=(x_0,\ldots,x_{N+1})\in \bT^{N+2}.
\end{equation} 
We construct an integer-valued, time inhomogeneous Markov chain $\{Y_j\}_{j=1}^N$, with 
\begin{equation}\label{e.6174}
\begin{aligned}
& \Pb_\x[Y_1=j_1]= \frac{\cZ_{1,0}(j_1+x_1,x_0)}{\G_{1,0}(x_1,x_0)},\\
&\Pb_\x[Y_2=j_2|Y_1=j_1]= \frac{\cZ_{2,1}(j_2+x_2,j_1+x_1)}{\G_{2,1}(x_2,x_1)},\\
&\ldots\\
&\Pb_\x[Y_N=j_N|Y_{N-1}=j_{N-1}]=\frac{\cZ_{N,N-1}(j_N+x_N,j_{N-1}+x_{N-1})}{\G_{N,N-1}(x_N,x_{N-1})},\\
&\Pb_\x[Y_{N+1}=j_{N+1}|Y_{N}=j_{N}]=\frac{\cZ_{s,N}(j_{N+1}+x_{N+1},j_{N}+x_{N})}{\G_{s,N}(x_{N+1},x_{N})}
\end{aligned}
\end{equation}
By \eqref{e.peZG} and \eqref{e.6174}, it is clear that $Y_{N+1}$ is a sum of independent random variables, for each fixed realization of the noise and $\x$. We rewrite it as 
\begin{equation}\label{e.defeta}
Y_{N+1}=\sum_{k=1}^{N+1} \eta_k, \quad\quad\mbox{ with } Y_0=0, \quad   \eta_k=Y_k-Y_{k-1},
\end{equation}
and one can interpret $\eta_k$ as the winding number accumulated during the time interval $[k-1,k]$ for $k=1,\ldots,N$, and $\eta_{N+1}$ corresponds to the winding number in $[N,s]$. 
We have 
\begin{equation}\label{e.laweta}
\begin{aligned}
&\Pb_\x[\eta_k=j]=\frac{\cZ_{k,k-1}(x_k+j,x_{k-1})}{\G_{k,k-1}(x_k,x_{k-1})}, 
\quad\quad j\in \Z, \quad k=1,\ldots,N,\\
&\Pb_\x[\eta_{N+1}=j]=\frac{\cZ_{s,N}(x_{N+1}+j,x_{N})}{\G_{s,N}(x_{N+1},x_{N})}, 
\quad\quad j\in \Z.
\end{aligned}
\end{equation}

With the Markov chain, one can write the summation in \eqref{e.7151} as 
\[
\begin{aligned}
&\sum_{j_1,\ldots,j_{N}} \frac{\cZ_{s,N}(j_{N+1}+x_{N+1},j_N+x_N)\prod_{k=1}^N \cZ_{k,k-1}(j_k+x_k,j_{k-1}+x_{k-1})  }{\G_{s,N}(x_{N+1},x_N)\prod_{k=1}^N \G_{k,k-1}(x_k,x_{k-1}) }\\
&=\Pb_{\mathbf{x}}[Y_{N+1}=j_{N+1}],
\end{aligned}
\]
where, to emphasize the dependence of the Markov chain on
$x_0,\ldots,x_{N+1}$, we have denoted the probability by $\Pb_{\mathbf{x}}$. 

Therefore, \eqref{e.7151} is rewritten as 
\begin{equation}\label{e.6172}
\begin{aligned}
\cZ_{s,0}(y,y')=\int_{\bT^N}\Pb_{\mathbf{x}}[Y_{N+1}=j_{N+1}] \G_{s,N}(x_{N+1},x_N)\prod_{k=1}^N \G_{k,k-1}(x_k,x_{k-1})    d\x_{1,N}.
\end{aligned}
\end{equation}

Now we introduce an additional notation: suppose that $f,g\in{\mathbb
  D}_c(\bT)$ - the set of continuous densities on the torus $\bT$. Define
\begin{equation}\label{e.muN3}
\mathscr{G}_N(\x;f, g ):=\frac{f(x_{N+1})\G_{s,N}(x_{N+1},x_N)\prod_{k=1}^N
  \G_{k,k-1}(x_k,x_{k-1})g(x_0)}{C_{N,0}(f,g)},
\end{equation}
with $\x=(x_0,\ldots,x_{N+1})$ and the normalization factor 
\begin{equation}
\label{GNM}
\begin{aligned}
C_{N,0}(f,g)&:=\int_{\bT^{N+2}}f(x_{N+1})\G_{s,N}(x_{N+1},x_N)\prod_{k=1}^N
  \G_{k,k-1}(x_k,x_{k-1})  g(x_0)  d\x_{0,N+1}\\
  &=\int_{\bT^2}f(x_{N+1})\G_{s,0}(x_{N+1},x_0)g(x_0)dx_0dx_{N+1}.
  \end{aligned}
\end{equation}
 For each realization of the random environment, one should view $\mathscr{G}_N(\x;f,g)$ as the joint density of the polymer on the cylinder, evaluated at 
\[
(0,x_0),(1,x_1),\ldots,(N,x_N),(s,x_{N+1}),
\] with $x_0$ and $x_{N+1}$ sampled from the densities $g,f$
respectively. For any $\mu,\nu\in \mathcal{M}_1(\bT)$ and $(x_1,\ldots,x_N)\in\bT^N$, we abuse the
notation and    write $\mathscr{G}_N(\x;\mu,\nu)$ as a measure
\begin{equation}\label{e.muN31}
   \mathscr{G}_N(\x;\mu,\nu):=\frac{\mu(dx_{N+1})\G_{s,N}(x_{N+1},x_N)\prod_{k=1}^N
  \G_{k,k-1}(x_k,x_{k-1})\nu(dx_0)}{C_{N,0}(\mu,\nu)}.
\end{equation}
In this case, $C_{N,0}(\mu,\nu)$ equals to
\[
C_{N,0}(\mu,\nu):=\int_{\bT^{N+2}}\G_{s,N}(x_{N+1},x_N)\prod_{k=1}^N
  \G_{k,k-1}(x_k,x_{k-1})  \nu(dx_0) d\x_{1,N}\mu(dx_{N+1}).
  \]

 Recall that the winding number $\W(s;\mu,\nu)$ has the representation given by \eqref{e.wsym}.  With the  Markov chain $\{Y_j\}_{j=1}^N$ and the above notation, one can express $\cZ_{s,0}(y,y')$ in terms of $ Y_{N+1}$ and $\mathscr{G}_N$. This leads to 
 \begin{lemma}
  We have \begin{equation}\label{e.windingnew}
  \cW(s;\mu,\nu)=-\int_{\bT^{N+2}}\big(x_{N+1}+\E_\x Y_{N+1}-x_0\big)\mathscr{G}_N(\x;\mu,\nu)d\x_{1,N},
  \end{equation}
where $\E_\x$ is the expectation w.r.t. $\Pb_\x$. 
\end{lemma}

\begin{remark}
On the r.h.s. of \eqref{e.windingnew}, the domain of integration is $\bT^{N+2}$, and this is because we have defined $\mathscr{G}_N(\x;\mu,\nu)$ as in \eqref{e.muN31} which includes $\mu(dx_{N+1})$ and $\nu(dx_0)$.
\end{remark}

\begin{proof}
Recall the expression for $\cW(s;\mu,\nu)$ obtained in \eqref{e.wsym}: 
\[
\cW(s;\mu,\nu)=-\frac{\int_{\R} \left(\int_{\bT} \cZ_{s,0}(y,y')(y-y')\nu(dy')\right) \mu(dy)}{\int_{\R} \left(\int_{\bT} \cZ_{s,0}(y,y')\nu(dy')\right) \mu(dy)}.
\]
For the numerator, using \eqref{e.6172} and \eqref{e.muN31}, we rewrite it as (with $y=j_{N+1}+x_{N+1}, y'=x_0$)
\[
\begin{aligned}
&\int_{\R} \left(\int_{\bT} \cZ_{s,0}(y,y')(y-y')\nu(dy')\right) \mu(dy)\\
&=C_{N,0}(\mu,\nu)\sum_{j_{N+1}\in\Z} \int_{\bT^2}\Pb_\x[Y_{N+1}=j_{N+1}] (j_{N+1}+x_{N+1}-x_0)\mathscr{G}_N(\x;\mu,\nu)dx_{1,N}\\
&=C_{N,0}(\mu,\nu)\int_{\bT^2}(\E_\x[Y_{N+1}]+x_{N+1}-x_0)\mathscr{G}_N(\x;\mu,\nu)dx_{1,N}.
\end{aligned}
\]
Similarly, it is easy to check that the denominator equals to $C_{N,0}(\mu,\nu)$. The proof is complete.
\end{proof}

\subsection{Exponential mixing} In this section, we start from the expression in \eqref{e.windingnew} and make use of the exponential mixing of the endpoint distribution of directed polymer on the cylinder to prove \eqref{e.Wsta}.

First, we need the following moment estimate
\begin{lemma}\label{l.mmestimate}
For any $p\in[1,\infty)$, we have 
\begin{equation}\label{e.641}
\sup_{s\in(0,2]}\sup_{x,y\in\bT} \EE \sum_{j\in\Z} |j|^p \frac{\cZ_{s,0}(x+j,y)}{\G_{s,0}(x,y)}\leq C_p.
\end{equation}
As a result, we conclude that
\begin{equation}\label{e.642}
\sup_{N\geq1} \sup_{\x}\max_{k=1,\ldots,N+1} \EE \E_\x |\eta_k|^p \leq C_p.
\end{equation}
\end{lemma}
Similar estimates can be found in \cite[Lemma 3.2, Lemma 4.1]{YGTK22}. The proof here covers more general cases.

\begin{proof}
First we state the following estimate, which can be proved using
rather standard techniques (e.g. see \cite[Theorem 2.6]{HK22}, the
proof in case of spatially periodic noise can be repeated verbatim): there exists $C>0$, only depending on $p\geq1$, such that for $s\in(0,2]$, $x,y\in\bT$ and $j\in\Z$, we have
\begin{equation}\label{e.positivemm}
 \EE \cZ_{s,0}(x+j,y)^p \leq Cq_s(x+j-y)^p.
\end{equation}
\blue{From \eqref{e.positivemm} and the triangle inequality, we have (with $\|\cdot\|_p$ denoting the $L^p(\Omega)$ norm)
\begin{equation}\label{e.positivemm1}
\|\G_{s,0}(x,y)\|_p=\|\sum_{j\in\Z} \cZ_{s,0}(x+j,y)\|_p \leq C \sum_{j\in \Z} q_s(x+j-y)=CG_s(x-y)
\end{equation}}
We also have the negative moment estimate by \cite[Corollary 4.8]{HK22}\footnote{The result in \cite{HK22} is for the equation on the whole space, but the proof applies verbatim to the torus case.}:
for any $p\in[1,\infty)$ there exists $C>0$ such that 
\begin{equation}\label{e.negativemm}
  \begin{aligned}
   & \EE\G_{s,0}(x,y)^{-p} \le 
    CG_s(x-y)^{-p},\quad x,y\in\bT,\,s\in(0,2].
    \end{aligned}
  \end{equation}


Applying the Cauchy-Schwarz inequality and the estimate
\eqref{e.positivemm} together with \eqref{e.negativemm}, we derive 
\[
\EE \sum_{j\in\Z} |j|^p \frac{\cZ_{s,0}(x+j,y)}{\G_{s,0}(x,y)} \leq C\sum_{j\in\Z} |j|^p\frac{q_s(x+j-y)}{G_s(x-y)}.
\]
Now we write
\[
\sum_{j\in\Z} |j|^p\frac{q_s(x+j-y)}{G_s(x-y)} \leq C+\sum_{|j|>100}|j|^p \frac{q_s(x+j-y)}{q_s(x-y)}
\]
and note that the r.h.s. is bounded, uniformly in $s\in(0,2]$ and $x,y\in\bT$. This completes the proof of \eqref{e.641}. To see why \eqref{e.642} holds, we only need to recall the definition of $\eta_k$ in \eqref{e.laweta} and apply \eqref{e.641}. 
\end{proof}

Next, we need the following lemma, which, roughly speaking, states that the dependence on $\eta_k$ on the two distributions $\mu$ and $\nu$ is exponentially small if $k\gg1$ and $N-k\gg1$ (recall that $N=\lf s\rf$).

\begin{lemma}\label{l.mixeta}
For any $p\in[1,\infty)$, there exist $C,\lambda>0$ such that 
\begin{equation}\label{e.mixeta}
\begin{aligned}
&\EE \left|\int_{\bT^{N+2}} \E_\x \eta_k [\mathscr{G}_N(\x;\mu_1,\nu_1)- \mathscr{G}_N(\x;\mu_1,\nu_2)]dx_{1,N}\right|^p \leq Ce^{-\lambda k},\\
&\EE \left|\int_{\bT^{N+2}} \E_\x \eta_k [\mathscr{G}_N(\x;\mu_1,\nu_1) -\mathscr{G}_N(\x;\mu_2,\nu_1)]dx_{1,N}\right|^p \leq C  e^{-\lambda(N-k)},
\end{aligned}
\end{equation}
for all $\mu_i,\nu_i\in\mathcal{M}_1(\bT)$, with $i=1,2$, and $k=1,\ldots,N+1$.
\end{lemma}

\begin{proof}
We will only consider the first inequality as the second one is proved in the same way. By Lemma~\ref{l.mmestimate}, we only need to consider the case of $N\gg1$ and $k\gg1$. 

\emph{The case of $1\le k\leq N$.} By definition \eqref{e.laweta}, we know that 
\[
\E_\x  \eta_k=\sum_j j\Pb_\x[\eta_k=j]=\sum_j j\frac{\cZ_{k,k-1}(x_k+j,x_{k-1})}{\G_{k,k-1}(x_k,x_{k-1})}
\]
which only depends on $x_{k-1},x_k$ and the noise in the time interval $[k-1,k]$. Using the forward and backward polymer endpoint densities, defined in \eqref{e.forwardbackward}, we can integrate out all $x-$variables except $x_{k-1},x_k$ and obtain
\begin{equation}\label{e.64etak}
\begin{aligned}
&\int_{\bT^{N+2}} \E_\x\eta_k\,\mathscr{G}_N(\x;\mu_1,\nu_i)dx_{1,N}\\
&=a_i^{-1}\int_{\bT^2} \E_\x\eta_k\,\rhob(s,\mu_1;k,x_k)\G_{k,k-1}(x_k,x_{k-1})\rhof(k-1,x_{k-1};0,\nu_i)dx_{k-1}dx_k,
\end{aligned}
\end{equation}
with \[
a_i:=\int_{\bT^2}\rhob(s,\mu_1;k,x_k)\G_{k,k-1}(x_k,x_{k-1})\rhof(k-1,x_{k-1};0,\nu_i)dx_{k-1}dx_k,\quad
i=1,2.
\]
Then one can write the difference as 
\[
\int_{\bT^{N+2}} \E_\x\eta_k\,[\mathscr{G}_N(\x;\mu_1,\nu_1)-\mathscr{G}_N(\x;\mu_1,\nu_2)]dx_{1,N}=E_1+E_2
\]
with
\[
\begin{aligned}
E_1=a_1^{-1}\int_{\bT^2}& \E_\x\eta_k\,\rhob(s,\mu_1;k,x_k)\G_{k,k-1}(x_k,x_{k-1})\\
&\times[\rhof(k-1,x_{k-1};0,\nu_1)-\rhof(k-1,x_{k-1};0,\nu_2)]dx_{k-1}dx_k,
\end{aligned}
\]
and
\[
\begin{aligned}
E_2=&\int_{\bT^2} \E_\x\eta_k\,\rhob(s,\mu_1;k,x_k)\G_{k,k-1}(x_k,x_{k-1})\rhof(k-1,x_{k-1};0,\nu_2)dx_{k-1}dx_k\\
&\times (a_1^{-1}-a_2^{-1}).
\end{aligned}
\]
It suffices to show that 
\begin{equation}\label{e.err1err2}
\EE |E_1+E_2|^p \leq Ce^{-\lambda k},
\end{equation}
for some $C,\lambda>0$ depending only on $p$.

We collect three facts which will be used in the proof:
\begin{itemize}
  \item[(i)] there exists $C$, depending only on $p\in[1,\infty)$ such that  
\[
\EE \sup_{x,y\in\bT} \G_{k,k-1}(x,y)^p+\EE \sup_{x,y\in\bT}\G_{k,k-1}(x,y)^{-p}
\leq C, 
\]
see \cite[Lemma 4.1]{GK21}. 
\item[(ii)] by Proposition~\ref{p.tk1}, there exists $C$, depending only on $p\in[1,\infty)$, such that  
\[
\EE \sup_{x\in\bT}
|\rhof(k,x;0,\nu_1)-\rhof(k,x;0,\nu_2)|^p \leq
Ce^{-\lambda k}, \quad\quad k\geq 10.
\]
\item[(iii)] there exists $C$, depending only on $p\in[1,\infty)$, such that
\[
\EE \sup_{x\in\bT,\nu\in\mathcal{M}_1(\bT)}
|\rhof(k,x;0,\nu)|^p \leq C, \quad\quad k\geq 10.
\]
\end{itemize}
 Using (i) and   Jensen's inequality, we conclude that for any
 $p\geq1$,
 $$
 \EE \big[a_i^p+a_i^{-p}\big] \leq C,\quad i=1,2.
 $$
Now applying the H\"older inequality and (ii), we have 
\[
\EE |E_1|^p \leq \sqrt{\EE a_1^{-2p}}\sqrt{\EE \int_{\bT^2}\rhob(s,\mu_1;k,x_k) |E_3|^{2p}dx_{k-1}dx_k},
\]
with
$$
E_3=\E_\x\eta_k\,\G_{k,k-1}(x_k,x_{k-1})[\rhof(k-1,x_{k-1};0,\nu_1)-\rhof(k-1,x_{k-1};0,\nu_2)].
$$
Note that $\rhob(s,\mu_1;k,\cdot)$ is independent of $E_3$ and we have the bound
\[
  \begin{aligned}
    & \EE \int_{\bT^2}\rhob(s,\mu_1;k,x_k) |E_3|^{2p}dx_{k-1}dx_k\\
    &
    =
    \int_{\bT^2}\EE\rhob(s,\mu_1;k,x_k) \EE|E_3|^{2p}dx_{k-1}dx_k \leq
    Ce^{-\lambda k},
    \end{aligned}
\]
thanks to (i), (ii) and Lemma \ref{l.mmestimate}.
This implies that  $\EE |E_1|^p\leq Ce^{-\lambda k}$. The analysis of $E_2$ is similar, so we do not repeat it here. Combine them together, we conclude the proof of \eqref{e.err1err2}.

\emph{The case of $k=N+1$.}  In this case we know that 
\[
\E_\x\eta_{N+1}=\sum_j j \frac{\cZ_{s,N}(x_{N+1}+j,x_N)}{\G_{s,N}(x_{N+1},x_N)},
\]
and for $i=1,2$, we have 
\begin{equation}\label{e.64etaN}
\begin{aligned}
&\int_{\bT^{N+2}} \E_\x\eta_{N+1}\,\mathscr{G}_N(\x;\mu_1,\nu_i)dx_{0,N+1}\\
&=a_i^{-1}\int_{\bT^2} \E_\x\eta_{N+1}\,\G_{s,N}(x_{N+1},x_N)\rhof(N,x_{N};0,\nu_i)dx_N\mu_1(dx_{N+1}),
\end{aligned}
\end{equation}
with 
\[
a_i=\int_{\bT^2} \G_{s,N}(x_{N+1},x_N)\rhof(N,x_{N};0,\nu_i)dx_N\mu_1(dx_{N+1}).
\]
Compared to \eqref{e.64etak}, the difference is we do not have the factor $\rho_{\rm b}$ in \eqref{e.64etaN}.
The rest of the proof is the same once we use (iii) and the following estimate (see e.g. \cite[Lemma B.2]{GK21})
\[
\sup_{x_{N+1}\in\bT} \left\{\EE \left|\int_{\bT} \G_{s,N}(x_{N+1},x_N) dx_N\right|^p + \EE  \left|\int_{\bT} \G_{s,N}(x_{N+1},x_N) dx_N\right|^{-p}\right\}\leq C.
\]
The proof is complete.
\end{proof}

A similar proof of the previous lemma leads to 
\begin{lemma}\label{l.mixendpoint}
There exist $C,\lambda>0$ such that 
\begin{equation}\label{e.mixendpoint}
\EE
\left|\int_{\bT^{N+2}}x_{N+1}[\mathscr{G}_N(\x;\mu_1,\nu_1)-\mathscr{G}_N(\x;\mu_1,\nu_2)]d\x_{1,N}\right|^p
\leq Ce^{-\lambda N},
\end{equation}
\begin{equation}\label{e.mixendpoint1}
\EE
\left|\int_{\bT^{N+2}}x_0[\mathscr{G}_N(\x;\mu_1,\nu_1)-\mathscr{G}_N(\x;\mu_2,\nu_1)]d\x_{1,N}\right|^p
\leq Ce^{-\lambda N}  ,
\end{equation}
 for any $\mu_i,\nu_i\in\mathcal{M}_1(\bT)$, $i=1,2$.
\end{lemma}

Using \eqref{e.defeta}, \eqref{e.windingnew} and  Lemma \ref{l.mixeta},
we obtain 
\begin{corollary}
  \label{cor011205-23}
    For any $p\in[1,\infty)$, there exists $C>0$ such that for all $s\geq0$,  
    \begin{equation}
      \label{011205-23}
\blue{\sup_{\mu_1,\mu_2,\nu_1,\nu_2\in{\mathcal M}_1(\bT)}\EE\big|\cW(s;\mu_1,\nu_1)-\cW(s,\mu_2,\nu_2)\big|^p\leq C.}
\end{equation}
  \end{corollary}
  
  \begin{remark}
  Recall that $\psi(s,y)=\cW(s;\delta_y,\delta_0)$ and $\phi(s,y)=\cW(s;\delta_y, \rho_{W})$. It is straightforward  to check that another consequence of Lemmas~\ref{l.mixeta} and \ref{l.mixendpoint} is the convergence in distribution of $\psi(s,y')-\psi(s,y)$ and $\phi(s,y')-\phi(s,y)$ as $s\to\infty$, for any $y',y\in\bT$. As a matter of fact, there is a one-force-one-solution principle: suppose we abuse the notation and let $\psi(s,y)$ (or $\phi(s,y)$) be the winding number of the polymer path in the time interval $[-K,s]$, then $\psi(s,y')-\psi(s,y)$ (or $\phi(s,y')-\phi(s,y)$) converges in $L^p(\Omega)$ as $K\to\infty$.
  \end{remark}
%
%

\subsection{Proof of Proposition~\ref{p.exerr}}
\label{sec3.4}

With Lemmas \ref{l.mixeta} and \ref{l.mixendpoint}, we can complete the proof of \eqref{e.Wsta} and conclude the proof of  Proposition~\ref{p.exerr}.


\begin{proof}[Proof of \eqref{e.Wsta}]
By Lemma~\ref{l.mmestimate}, it suffices to consider those $N\gg1$.
By \eqref{e.windingnew} and \eqref{e.defeta}, we can first bound the difference as  
\[
\begin{aligned}
&\big|[\cW(s;\mu_1,\nu_1)-\cW(s,\mu_2,\nu_1)]-[\cW(s;\mu_1,\nu_2)-\cW(s;\mu_2,\nu_2)]\big|\\
&\leq E_1+E_2+E_3,
\end{aligned}
\]
with 
\[
\begin{aligned}
E_1&=\sum_{i=1}^2\left|\int_{\bT^{N+2}}x_{N+1}[\mathscr{G}_N(\x;\mu_i,\nu_1)-\mathscr{G}_N(\x;\mu_i,\nu_2)]d\x_{1,N}\right|,\\
E_2&=\sum_{i=1}^2  \left|\int_{\bT^{N+2}}x_0[\mathscr{G}_N(\x;\mu_1,\nu_i)-\mathscr{G}_N(\x;\mu_2,\nu_i)]d\x_{1,N}\right|,\\
E_3&=\sum_{i=1}^2\left|\sum_{N+1\ge k\geq N/2}\int_{\bT^{N+2}} \E_\x \eta_k [\mathscr{G}_N(\x;\mu_i,\nu_1)- \mathscr{G}_N(\x;\mu_i,\nu_2)]dx_{1,N}\right|,\\
&+\sum_{i=1}^2\left|\sum_{1\le k<N/2}\int_{\bT^{N+2}} \E_\x \eta_k [\mathscr{G}_N(\x;\mu_1,\nu_i)- \mathscr{G}_N(\x;\mu_2,\nu_i)]dx_{1,N}\right|.
\end{aligned}
\]
Then it suffices to apply the estimates in Lemmas~\ref{l.mixeta} and \ref{l.mixendpoint} to complete the proof.
\end{proof}

\section{Passive scalar in Burgers flow and time reversal}
\label{s.scalar}

Recall that in Section~\ref{s.co}, we have used the Clark-Ocone formula to write the \blue{quantity of interest}, the quenched mean $X_t$, as an It\^o integral 
\begin{equation}\label{e.co651}
X_t=\int_{\R} x\rho(t,x)dx=\beta \int_0^t\int_{\bT} \EE[I_{1,t}(s,y)-I_{2,t}(s,y)|\F_s] \xi(s,y)dyds,
\end{equation}
with 
\begin{equation}\label{e.co652}
I_{1,t}(s,y)-I_{2,t}(s,y)=\rho_{\rm m}(t,-|s,y)\int_{\bT}[\psi(s,y')-\psi(s,y)]\rho_{\rm m}(t,-|s,y')dy'.
\end{equation}
In Section~\ref{s.winding}, we have derived the estimate in
Proposition~\ref{p.exerr} which will eventually enable us to replace
$\psi(s,y')-\psi(s,y)$ by its ``stationary'' version
$\phi(s,y')-\phi(s,y)$. Therefore, to have an explicit expression for
the variance of $X_t$, it remains yet to understand the joint distribution of $\rho_{\rm m}(t,-|s,\cdot)$ and $\nabla \phi(s,\cdot)$. It turns out to be the major challenge and is the goal of this section.  

\blue{The rest of this section is divided into three parts. In Section~\ref{s.windingcorrector}, we use the Girsanov transformation to reduce the study of the directed polymer to a diffusion in a \emph{non-Markovian} drift. By a standard homogenization argument, the quenched mean of the polymer endpoint turns out to be the corrector from a formal two-scale expansion. In Section~\ref{s.fkreversal}, we further study the gradient of the corrector and draw the connection to another diffusion with a \emph{Markovian} drift. Using the time-reversal anti-symmetry between the aforementioned two drifts, we form a link between the Gibbsian and the Markovian settings, which eventually enable us to derive the joint law of the gradient of the corrector and the drift, see Proposition~\ref{p.keyP} below. This may be compared to the study of geodesics and competing interfaces in the context of last passage percolation \cite{Sep20}. On the technical side, all these will be done on the level of approximations, since the drifts are distribution valued. In Section~\ref{s.singulardiffusion} we borrow the tools from singular diffusion to pass to the limit.}


\subsection{Winding number, diffusion in random environment, corrector} 
\label{s.windingcorrector}

 The goal of this section is to utilize the well-known connection
 between the directed polymer and the passive scalar in a random
 drift, which solves the stochastic Burgers equation, and to  rewrite
 the difference $\phi(s,y')-\phi(s,y)$ in terms of the solution of a
 Fokker-Planck equation with another, different but related,
 drift. One technical difficulty is that, since the random environment
 for the directed polymer is the spacetime white noise, the random
 drift for the passive scalar is distribution valued. \blue{For this
   reason we start with an approximation.}

Recall that $\phi$ was defined in \eqref{e.defphi}. For any $t>0, x\in\bT$, we define the $\eps-$approximation of $\phi$ as
\begin{equation}\label{e.defphieps}
\begin{aligned}
\phi^\eps(t,x)&=\frac{\int_{\R} \cZ^\eps_{t,0}(x,x')(x'-x)e^{\beta W(x')}dx'}{\int_{\R} \cZ^\eps_{t,0}(x,x')e^{\beta W(x')}dx'}\\
&=\frac{\E_B[e^{\beta\int_0^t \xi^\eps(t-\ell,B_\ell)d\ell-\frac12\beta^2 R^\eps(0)t}e^{\beta W(B_t)}(B_t-x)|B_0=x]}{\E_B[e^{\beta\int_0^t \xi^\eps(t-\ell,B_\ell)d\ell-\frac12\beta^2 R^\eps(0)t}e^{\beta W(B_t)}|B_0=x]}.
\end{aligned}
\end{equation}
For any $s<t, x,y\in \bT$, define the midpoint density of the ``stationary'' polymer at $(s,y)$ by 
\begin{equation}\label{e.defrhoeps}
\rho^\eps_{\rm m}(t,x|s,y)=\frac{\E_B[e^{\beta\int_0^t \xi^\eps(t-\ell,B_\ell)d\ell-\frac12\beta^2 R^\eps(0)t}e^{\beta W(B_t)}\delta(\dot{B}_{t-s}-y)|B_0=x]}{\E_B[e^{\beta\int_0^t \xi^\eps(t-\ell,B_\ell)d\ell-\frac12\beta^2 R^\eps(0)t}e^{\beta W(B_t)} |B_0=x]},
\end{equation}
where we recall that $\dot{B}_s$ denotes the fractional part of ${B}_s$. It is clear that 
\[\int_{\bT} \rho^\eps_{\rm m}(t,x|s,y)dy=1.
\]
If we view the polymer path as lying on the cylinder, starting from $(t,x)$ and going backward in time with the terminal potential given by $\beta W$, then $\rho^\eps_{\rm m}(t,x|s,\cdot)$ is precisely the quenched density at time $s$.  Using the propagator of SHE (with mollified noise $\xi^\eps$), we can rewrite $\rho^\eps_{\rm m}$ as 
\begin{equation}\label{e.defrhoepsnew}
\begin{aligned}
\rho^\eps_{\rm m}(t,x|s,y)&=\frac{ \sum_{n\in\Z} \cZ^\eps_{t,s}(x,y+n)\int_{\R}\cZ^\eps_{s,0}(y+n,x')e^{\beta W(x')}dx'}{\int_{\R} \cZ^\eps_{t,0}(x,x')e^{\beta W(x')}dx'}\\
&=   \frac{ \G^\eps_{t,s}(x,y)\int_{\bT}\G^\eps_{s,0}(y,x')e^{\beta W(x')}dx'}{\int_{\bT} \G^\eps_{t,0}(x,x')e^{\beta W(x')}dx'}.
\end{aligned}
\end{equation}

\begin{remark}
At this point, it is worth pointing out the difference between the two midpoint densities, $\rho_{\rm m}(t,-|s,y)$ and $\rho_{\rm m}^\eps(t,x|s,y)$, defined in \eqref{e.defrho} and \eqref{e.defrhoeps} respectively -- one should not confuse them with each other. Apart from the different noise $\xi$ and $\xi^\eps$, the prescribed distribution of two endpoints are different: for $\rho_{\rm m}(t,-|s,y)$, the endpoint at $t$ is distributed according to Lebesgue measure on $\bT$ and the endpoint at $0$ is fixed at $0$; for $\rho_{\rm m}^\eps(t,x|s,y)$, the endpoint at $t$ is fixed at $x$ and the endpoint at $0$ is distributed according to the density $\rho_W(x)=e^{\beta W(x)}/\int_{\bT} e^{\beta W(x')}dx'$.
\end{remark}

Define $Z^\eps_W(t,x):=\int_{\R} \cZ^\eps_{t,0}(x,y)e^{\beta W(y)}dy$, which solves
\begin{equation}\label{e.Zeps}
\partial_tZ^\eps_W=\frac12\Delta Z^\eps_W+\beta Z^\eps_W\xi^\eps, \quad\quad Z^\eps_W(0,x)=e^{\beta W(x)},
\end{equation} and define 
\begin{equation}\label{e.hueps}
h^\eps=\log Z^\eps_W, \quad\quad u^\eps=\nabla h^\eps.
\end{equation}
\blue{Here one should keep in mind that both $h^\eps$ and $u^\eps$
  depend on $W$, and we only have kept it implicit to make
  the notation less burdensome.}

From now on, we fix some $t>0$ and $x\in\bT$.  Consider the SDE
\begin{equation}\label{e.rwre}
d\cX_s^\eps=u^\eps(t-s,\cX_s^\eps)ds+dB_s, \quad\quad \cX_0^\eps=x.
\end{equation}
We emphasize that the above is a simplified notation where the dependence on $(t,x)$ is omitted.

\begin{lemma}
  \label{lm012104-23}
  The quenched law of $\{\cX_s^\eps\}_{0\le s\le t}$ coincides with that of
  the path $\{B_s\}_{0\le s\le t}$ of the random polymer under the measure
  $$
  \frac{1}{Z^\eps_W(t,x)}\exp\left\{\beta\int_0^t\xi^\eps(t-s,
                 B_{s})ds-\frac{\beta^2}{2}
                 R^\eps(0)t+\beta W(B_t)\right\}\mathbb W_t(dB),
               $$
               where $\mathbb W_t$ is the Wiener measure on $C[0,t]$ with $B_0=x$.
  \end{lemma}
 \proof The result is well-known, and the proof is rather standard, 
 so we only present it here for the convenience of readers. 
Recall that $h^\eps=\log Z_W^\eps$. Consider $Y_s=h^\eps(t-s,B_s)$, with $B$ a standard Brownian motion starting at
$x$. Applying Ito formula, 
 we have 
\[
\begin{aligned}
dY_s&=(-\partial_t h^\eps+\frac12\Delta h^\eps)ds+\nabla h^\eps dB_s\\
&=\left(\nabla h^\eps(t-s,B_s)dB_s-\frac12|\nabla h^\eps(t-s,B_s)|^2ds\right)-\beta \xi^\eps(t-s,B_s)ds+\frac12\beta^2 R^\eps(0)ds,
\end{aligned}
\]
which implies that (recall that $h^\eps(0,x)=\beta W(x)$)
\[
\begin{aligned}
\beta W(B_t)=&h^\eps(t,x)+\int_0^t \left(\nabla h^\eps(t-s,B_s)dB_s-\frac12|\nabla h^\eps(t-s,B_s)|^2ds\right)\\
&-\beta\int_0^t \xi^\eps(t-s,B_s)ds+\frac12\beta^2R^\eps(0)t.
\end{aligned}
\]
So one can write 
\[
\begin{aligned}
&\frac{1}{Z^\eps_W(t,x)}\exp\left\{\beta \int_0^{t} \xi^\eps(t-s,B_s)ds-\frac12\beta^2 R^\eps(0)t+\beta W(B_t)\right\}\\
&=\exp\left\{\int_0^{t} \nabla h^\eps(t-s,B_s)dB_s-\frac12\int_0^{t} |\nabla h^\eps(t-s,B_s)|^2 ds\right\}.
\end{aligned}
\]
It remains to apply the Girsanov theorem to complete the proof.
\qed


Using the previous lemma, we can express the winding number $\phi^\eps$ in terms of the diffusion $\cX^\eps$:
\begin{lemma}\label{l.phiepsPro}
For any $t>0, x\in\bT$, we have 
\begin{equation}\label{011606-23}
\begin{aligned}
\phi^\eps(t,x)=\E_B [\cX_t^\eps]-x&=\int_0^t \E_B u^\eps(t-s,\cX_s^\eps)ds\\
&=\int_0^t\int_{\bT} u^\eps(t-s,y)\rho^\eps_{\rm m}(t,x|t-s,y)dyds.
\end{aligned}
\end{equation}
\end{lemma}
\begin{proof}
First, from \eqref{e.defphieps} and Lemma~\ref{lm012104-23}, we know that $\phi^\eps(t,x)=\E_B \cX_t^\eps-x$.  Then, from the SDE \eqref{e.rwre}, we have
\[
\cX_t^\eps=x+\int_0^t u^\eps(t-s,\cX_s^\eps)ds+B_t.
\] 
Taking expectation on $B$ yields the second equality in \eqref{011606-23}.

Thanks to Lemma \ref{lm012104-23}, the quenched density of $\cX_s^\eps$ on the torus is $\rho^\eps_{\rm m}(t,x|t-s,\cdot)$, so the third equality  in \eqref{011606-23} follows,
which completes the proof.
\end{proof}

Define
\begin{equation}\label{e.defg}
g^\eps(t,x)=\nabla(x+\phi^\eps(t,x))=1+\nabla \phi^\eps(t,x),
\end{equation}
which will play a crucial role later. Next we show that   $g^\eps$ solve a Fokker-Planck equation with the random drift given by $u^\eps$:
\begin{lemma}\label{l.geps}
We have: $\phi^\eps$ solves
\begin{equation}\label{e.eqphieps}
\partial_t \phi^\eps=\tfrac12\Delta\phi^\eps+u^\eps\nabla \phi^\eps+u^\eps, \quad\quad \phi^\eps(0,x)=0,
\end{equation}
and $g^\eps$ solves
\begin{equation}\label{e.eqgeps}
\partial_t g^\eps=\tfrac12\Delta g^\eps+\nabla (u^\eps g^\eps), \quad\quad g^\eps(0,x)=1.
\end{equation}
\end{lemma}

\begin{proof}
Consider the solution to \eqref{e.eqphieps}, with a bit of an abuse of
notation we still denote it by $\phi^\eps$. Applying the It\^o formula
for the process  $Y_s=\phi^\eps(t-s,\cX_s^\eps)$, we obtain  
\[
\begin{aligned}
dY_s&=(-\partial_t \phi^\eps+\frac12\Delta \phi^\eps +u^\eps\nabla \phi^\eps ) ds+\nabla \phi^\eps dB_s\\
&=-u^\eps(t-s,\cX_s^\eps)ds+\nabla \phi^\eps(t-s,\cX_s^\eps)dB_s,
\end{aligned}
\]
which implies that (recall that $\cX_0^\eps=x$)
\begin{equation}\label{e.phiepsde}
0=\phi^\eps(t,x) -\int_0^t u^\eps(t-s,\cX_s^\eps)ds+\int_0^t \nabla\phi^\eps(t-s,\cX_s^\eps)dB_s.
\end{equation}
Taking expectation with respect to $B$ on both sides shows that 
\[
\phi^\eps(t,x)=\E_B \int_0^t u^\eps(t-s,\cX_s^\eps)ds.
\]
Applying Lemma~\ref{l.phiepsPro}, we complete the proof of \eqref{e.eqphieps}.
Since $\phi^\eps(t,x)$ is smooth in the $x-$variable and by definition $g^\eps=1+\nabla\phi^\eps$, we take derivatives in $x$ on both sides of \eqref{e.eqphieps} to obtain \eqref{e.eqgeps}.
\end{proof}

   The equation \eqref{e.eqphieps} for $\phi^\eps$ can be viewed as the
  corrector equation for the diffusion in random environment
  $\cX_s^\eps$ (it is the $\eps-$version of \eqref{e.eqphiIntro}).  Namely, using $\phi^\eps$ (see \eqref{e.phiepsde}), we
  can decompose the drift term in \eqref{e.rwre} as
\[
\int_0^t u^\eps(t-s,\cX_s^\eps)ds=\phi^\eps(t,x)+\int_0^t \nabla\phi^\eps(t-s,\cX_s^\eps)dB_s,
\]
so the process itself can be written as 
\begin{equation}\label{e.Xmade}
\begin{aligned}
\cX_t^\eps&=x+\int_0^t u^\eps(t-s,\cX_s^\eps)ds+B_t\\
&=x+\phi^\eps(t,x)+\int_0^t g^\eps(t-s,\cX_s^\eps)dB_s.
\end{aligned}
\end{equation}
In other words, to study the fluctuations of the  directed polymer endpoint, by Lemma~\ref{lm012104-23}, it suffices to consider the process $\cX^\eps$, and the decomposition in \eqref{e.Xmade} is the start of the standard homogenization argument. \blue{Nevertheless, unlike the ``static'' homogenization where the corrector
(constituting the so called  co-boundary term) does not
contribute on the diffusive scale, see e.g. the models in \cite[Chapter 9]{KLO12} and the references cited there,  in our dynamic setting with a compressible drift, the
$\phi^\eps$ term contributes to the limiting Gaussian distribution. Actually, the main challenge  of the paper is to quantify this contribution.} 

\subsection{Fokker-Planck equation and time reversal}
\label{s.fkreversal}

Recall that $\phi^\eps$ was defined in \eqref{e.defphieps}, and as $\eps\to0$ it converges to 
\begin{equation}
  \label{phi}
\phi(t,x)=\frac{\int_{\R} \cZ_{t,0}(x,x')(x'-x)e^{\beta W(x')}dx'}{\int_{\R} \cZ_{t,0}(x,x') e^{\beta W(x')}dx'}
\end{equation}
In addition, we have $Z_W^\eps, h^\eps,u^\eps$ converge to $Z_W,h,u$ respectively, with 
\begin{equation}\label{e.defZW}
\begin{aligned}
&Z_W(t,x)=\int_{\R} \cZ_{t,0}(x,y)e^{\beta W(y)}dy,\\
 &h=\log Z_W,  \quad\quad u=\nabla h.
\end{aligned}
\end{equation}
Thus, $h,u$ solve the KPZ, stochastic Burgers equation, respectively, that start from stationarity.  Similarly, the dependence of $h,u$ on $W$ was kept implicit.

Our goal is to derive the joint distribution of $(u(s,\cdot),\nabla \phi(s,\cdot))$ for large $s\gg1$, which was  needed in the Clark-Ocone representation (see \eqref{e.co651} and \eqref{e.co652}). Here $u(s,\cdot)$ will appear in the term $\rho_{\rm m}(t,-|s,y)$, and, once we approximate $\psi$ by $\phi$, the other term can be written as  $\phi(s,y')-\phi(s,y)=\int_y^{y'}\nabla\phi(s,z)dz$. It is worth pointing out that $\nabla\phi(s,\cdot)$ and $u(s,\cdot)$ are distribution-valued processes and what we actually need is the function-valued  ``increment processes'' $\phi(s,y)-\phi(s,0)$ and $h(s,y)-h(s,0)=\int_0^yu(s,z)dz$.

Since $\phi$ can be approximated by $\phi^\eps$ and the increments of $\phi^\eps$ is related to $g^\eps$ through
\[
\phi^\eps(s,y')-\phi^\eps(s,y)=\int_y^{y'}\nabla\phi^\eps(s,z)dz=\int_y^{y'}(g^\eps(s,z)-1)dz,
\]
it suffices to study the joint distribution of $u^\eps(s,\cdot)$ and $g^\eps(s,\cdot)$. On the other hand, we know from Lemma~\ref{l.geps} that $g^\eps$ solves the Fokker-Planck equation with the drift given by $u^\eps$:
\begin{equation}\label{e.eqgeps1}
\partial_t g^\eps=\frac12\Delta g^\eps+\nabla (u^\eps g^\eps), \quad\quad g^\eps(0,x)=1.
\end{equation}
Thus, the problem reduces to studying the above equation and the joint distribution of the coefficient and the solution. For a Fokker-Planck equation driven by a general random drift, it does not seem likely that one could obtain such explicit descriptions. However, in our case of the Burgers drift, there is an amazing time
  reversal anti-symmetry (see \cite[Proposition 1.1]{BQS11}):
  \begin{proposition}
    \label{prop011105-23}
    Fix any $t>0$. Then
\begin{equation}\label{e.burgersAsym}
\{u(s,\cdot)\}_{s\in[0,t]}\stackrel{\text{law}}{=}\{-u(t-s,\cdot)\}_{s\in[0,t]}.
\end{equation}
\end{proposition}

\blue{The above relation \eqref{e.burgersAsym} only holds for $u$, but not for $u^\eps$. Nevertheless, this inspires us to define another diffusion driven by a time reversed drift and connect $g^\eps$ in \eqref{e.eqgeps1} to its Fokker-Planck equation.}  

Before that, we first note that the  solution to \eqref{e.eqgeps1} is periodic in space, so we view it as the
Fokker-Planck equation satisfied by the density of the following diffusion in random environment, with $\bT$ being the state space:
\begin{equation}\label{e.rwre1}
d\cY_s^\eps=-u^\eps(s,\cY_s^\eps)ds+dB_s, \quad\quad \cY_0^\eps \sim m_\bT,
\end{equation}
where ``$\cY_0^\eps\sim m_{\bT}$'' means that $\cY_0^\eps$ is sampled
from the Lebesgue measure on $\bT$, denoted by $m_\bT$. It is clear
that for any $s>0$, $g^\eps(s,\cdot)$ is the quenched density of
$\cY_s^\eps$ on $\bT$,   
\[
g^\eps(s,y)\ge0\quad\mbox{and}\quad \int_{\bT} g^\eps(s,y)dy=1.
\] 

Inspired by the time reversal anti-symmetry in \eqref{e.burgersAsym}, we consider another diffusion in random environment, which is related to the directed polymer in light of Lemma~\ref{lm012104-23}: for fixed $t>0$, let $\tilde{\cY}^\eps$ solve
\begin{equation}\label{e.rwre2}
d\tilde{\cY}_s^\eps=u^\eps(t-s,\tilde{\cY}_s^\eps)ds+dB_s, \quad\quad \tilde{\cY}_0^\eps\sim m_\bT.
\end{equation}
Let $\tilde{g}^\eps(t;s,\cdot)$ be the quenched density of
$\tilde{\cY}_s^\eps$ (again on $\bT$ rather than $\R$). Comparing
$\tilde{\cY}_s^\eps$ to $\cX_s^\eps$ defined in \eqref{e.rwre}, there are two differences: (i) 
 the distribution of the starting point, (ii) we
have chosen  $\R$ as the state space for $\cX_s^\eps$, while the state space for $ \tilde{\cY}_s^\eps$ is $\bT$. Here $t>0$ is fixed, and we have used the simplified notations $\cX_s^\eps,\tilde{\cY}_s^\eps$,  where the dependence on $t$ was again omitted.

We have the following lemma which relates $\tilde{g}^\eps$ to the
midpoint density of a directed polymer:
\begin{lemma}\label{l.tildegepsPro}
For any $s\in(0,t]$ and $y\in \bT$, we have 
\[
\begin{aligned}
\tilde{g}^\eps(t;s,y)&=\int_\bT \rho^\eps_{\rm m}(t,x|t-s,y)dx\\
&=\int_{\bT}\left( \frac{ \G^\eps_{t,t-s}(x,y)\int_{\bT}\G^\eps_{t-s,0}(y,x')e^{\beta W(x')}dx'}{\int_{\bT} \G^\eps_{t,0}(x,x')e^{\beta W(x')}dx'}\right) dx,
\end{aligned}
\]
where $\rho^\eps_{\rm m}$ was given in \eqref{e.defrhoepsnew}.
\end{lemma}

\begin{proof}
By Lemma~\ref{lm012104-23}, we know that if
$\tilde{\cY}_0^\eps=x$, then the quenched  density of
$\tilde{\cY}_s^\eps$ is given by $\rho^\eps_{\rm m}(t,x|t-s,y)$
and this actually completes the proof since $\tilde{\cY}_0^\eps$  is sampled from the Lebesgue measure on $\bT$.
\end{proof}

We are particularly interested in $\tilde{g}^\eps(t;t,\cdot)$, which takes the form
\[
\tilde{g}^\eps(t;t,y)=\int_{\bT}\left( \frac{  \G_{t,0}^\eps(x,y) e^{\beta
      W(y)}}{\int_{\bT} \G_{t,0}^\eps(x,x') e^{\beta W(x')}dx'}\right) dx, \quad\quad y\in \bT.
\]
Define the random density \begin{equation}\label{e.deftildeg}
\tilde{g}(t,y)=\int_{\bT}\left( \frac{  \G_{t,0}(x,y) e^{\beta
      W(y)}}{\int_{\bT} \G_{t,0}(x,x')e^{\beta W(x')}dx'}\right) dx,
\quad\quad  y\in \bT.
\end{equation}
The following result is rather standard, and the proof follows an
approximation argument similar to that of \cite[Theorem 2.2]{BC95}.
\begin{lemma}\label{c.contildeg}
For any $t>0$, we have
\[
\begin{aligned}
\tilde{g}^\eps(t;t,\cdot)\to \tilde{g}(t,\cdot),\quad\mbox{ as}\quad \eps\to0
\end{aligned}
\]
 in $C(\bT)$ in probability. 
\end{lemma}

Define 
\begin{equation}\label{e.deftildephi}
\tilde{\phi}(t,y)=\int_0^y [\tilde{g}(t,y')-1]dy', \quad\quad y\in \bT.
\end{equation}
Here is the main result of this section (recall that $\nabla h^\eps=u^\eps$ and $\nabla\phi^\eps=g^\eps-1$). 
\begin{proposition}\label{p.keyP}
Fix any $t>0$, as $\eps\to0$, we have the following weak convergence on $C(\bT)\times C(\bT)$:
\begin{equation}\label{e.keyConver}
(h^\eps(t,\cdot)-h^\eps(t,0),\phi^\eps(t,\cdot)-\phi^\eps(t,0))\Rightarrow (-\beta W(\cdot),\tilde{\phi}(t,\cdot)).
\end{equation}
As a result,
\begin{equation}\label{e.equallaw}
(h(t,\cdot)-h(t,0), \phi(t,\cdot)-\phi(t,0))\stackrel{\text{law}}{=}(-\beta W(\cdot),\tilde{\phi}(t,\cdot)).
\end{equation}
\end{proposition}

The above result comes from the time reversal anti-symmetry
 \eqref{e.burgersAsym}, which only holds in the limiting case of
 $\eps=0$. To understand where \eqref{e.equallaw} comes from, let us
 pretend the anti-symmetry also holds for $\eps>0$: if
  $\{u^\eps(s,\cdot)\}_{s\in[0,t]}\stackrel{\text{law}}{=}\{-u^\eps(t-s,\cdot)\}_{s\in[0,t]}$,
  then by comparing the two diffusions described by \eqref{e.rwre1}
  and \eqref{e.rwre2} respectively, it is straightforward to conclude
  that $(u^\eps(t,\cdot),
  g^\eps(t,\cdot))\stackrel{\text{law}}{=}(-u^\eps(0,\cdot),\tilde{g}^\eps(t;t,\cdot))$, which can be seen as an $\eps-$version of \eqref{e.equallaw}. 
  Although this argument does not apply to   $\eps>0$, as the time reversal anti-symmetry for the Burgers solution  holds only in the limiting case of $\eps=0$, it is natural to expect \eqref{e.keyConver} to hold.
The proof of Proposition~\ref{p.keyP} uses the well-developed results
for singular diffusions with distribution-valued drifts, which we prove in the next section.

\blue{\begin{remark}
The relation \eqref{e.equallaw} gives the explicit joint law of $(u(t,\cdot),\nabla \phi(t,\cdot))$. On the formal level, $g=1+\nabla\phi$ solves the Fokker-Planck equation $\partial_t g=\tfrac12\Delta g+\nabla (ug)$, so if we consider the singular diffusion $d\cY_s=-u(s,\cY_s)ds+dB_s$, it actually leads to the explicit invariant measure of the process of ``the environment seen from the particle''.
\end{remark}}

\subsection{Singular diffusion and proof of Proposition~\ref{p.keyP}}
\label{s.singulardiffusion}
Fix $t>0$ in this section. Recall that
\[
\begin{aligned}
\phi^\eps(t,x)&=\frac{\int_{\R} \cZ^\eps_{t,0}(x,x')(x'-x)e^{\beta W(x')}dx'}{\int_{\R} \cZ^\eps_{t,0}(x,x')e^{\beta W(x')}dx'}\\
&=\frac{\int_{\R} \cZ^\eps_{t,0}(x,x')x'e^{\beta W(x')}dx'}{\int_{\R}
  \cZ^\eps_{t,0}(x,x')e^{\beta W(x')}dx'}-x,\\
h^\eps(t,x)&=\log  \int_{\R} \cZ^\eps_{t,0}(x,x')e^{\beta W(x')}dx',
\end{aligned}
\]
and $\phi,h$ are of the same form as $\phi^\eps,h^\eps$, with $\cZ^\eps$ replaced by $\cZ$.
Therefore, $\phi^\eps,h^\eps$ are both written as smooth functionals of functions of the form 
$
\int_{\R} \cZ^\eps_{t,0}(x,x')f(x')dx'$ for some continuous function
$f$, namely, the solution to SHE driven by $\xi^\eps$ started from
$f$.  Thus, by a standard argument (see e.g. \cite[Theorem 2.2]{BC95}) we know that, for each $t>0$, 
\[
(h^\eps(t,\cdot),\phi^\eps(t,\cdot))\Rightarrow (h(t,\cdot),\phi(t,\cdot)),
\]
weakly in $C(\bT)\times C(\bT)$.

With the above weak convergence, to prove Proposition~\ref{p.keyP}, it suffices to show the convergence of finite dimensional distributions in \eqref{e.keyConver}. In other words, take any $n\geq1$ and $x_1,\ldots,x_n\in\bT$, we aim at proving the convergence in distribution of 
\begin{equation}\label{e.confinite}
(X_1^\eps,\ldots,X_n^\eps,Y_1^\eps,\ldots,Y_n^\eps)\Rightarrow (X_1,\ldots,X_n,Y_1,\ldots,Y_n)
\end{equation}
as $\eps\to0$, with 
\begin{equation}\label{e.defXepsYeps}
\begin{aligned}
&X_i^\eps=h^\eps(t,x_i)-h^\eps(t,0), \quad \quad X_i=-\beta W(x_i),\\
&Y_i^\eps=\phi^\eps(t,x_i)-\phi^\eps(t,0),  \quad\quad Y_i=\tilde{\phi}(t,x_i).
\end{aligned}
\end{equation}

Before going to  the proof of \eqref{e.confinite}, we first introduce a few more notations. Denote $\zeta=(\xi,W)$ which represents the underlying random realization. For each realization of $\zeta$ (hence each realization of the drift $u^\eps$), let $\bbQ_t^\eps,\tilde{\bbQ}_t^\eps$ be the probability measures on $C([0,t],\bT)$ induced by $\{\cY_s^\eps\}_{s\in[0,t]}, \{\tilde{\cY}_s^\eps\}_{s\in[0,t]}$ respectively.  Namely, $\bbQ_t^\eps,\tilde{\bbQ}_t^\eps$ are the quenched probability measures for the diffusions in random environments, and they depend on $\eps>0$ and also on the random realization $\zeta$. Let $\mathcal{M}_1(C([0,t],\bT))$ denote the set of probability measures on $C([0,t],\bT)$, equipped with the topology of weak convergence. One can view $\bbQ_t^\eps,\tilde{\bbQ}_t^\eps$ as random variables taking values in $\mathcal{M}_1(C([0,t],\bT))$. Define 
\begin{equation}\label{e.defU}
\mathscr{U}_t=\{-u(s,x)\}_{s\in[0,t],x\in\bT}, \quad\quad \tilde{\mathscr{U}}_t=\{u(t-s,x)\}_{s\in[0,t],x\in\bT},
\end{equation}
which represent the drift for the two diffusions in random environment (in the limit of $\eps=0$), and we view them as random variables taking values $C([0,t],C^{-\alpha}(\bT))$ for some $\alpha>1/2$.

Next, we state a result which is borrowed from the singular diffusion literature and will play a crucial role in our analysis.

\begin{proposition}\label{p.sdiff}
As $\eps\to0$, $\bbQ_t^\eps$ and $\tilde{\bbQ}_t^\eps$ \blue{converge weakly} almost surely. In other words, there exist $\bbQ_t,\tilde{\bbQ}_t$, which are random variables taking values in $\mathcal{M}_1(C([0,t],\bT))$, such that for almost every $\zeta$,  we have $\bbQ_t^\eps\Rightarrow \bbQ_t$ and $\tilde{\bbQ}_t^\eps\Rightarrow \tilde{\bbQ}_t$ as $\eps\to0$, \blue{where ``$\Rightarrow$'' represents the weak convergence of measures.} In addition, we have 
\begin{equation}\label{e.identitylaw}
(\mathscr{U}_t,\bbQ_t)\stackrel{\text{law}}{=}(\tilde{\mathscr{U}}_t,\tilde{\bbQ}_t).
\end{equation}
\end{proposition}

The proof of Proposition~\ref{p.sdiff} will be presented in Section~\ref{s.sd} of the Appendix. We will first use it to complete the proof of \eqref{e.confinite}, which is divided into several steps. 

\begin{proof}[Proof of \eqref{e.confinite}]
Throughout the proof, fix $n\geq1$ and $x_1,\ldots,x_n\in\bT$. Let $\{\omega_s\}_{s\in[0,t]}$ be the canonical process of $C([0,t],\bT)$.

\emph{Step 1.} The goal is to show that under $\tilde{\bbQ}_t$, the random variable $\omega_t$ has the density given by $\tilde{g}(t,\cdot)$ defined in \eqref{e.deftildeg}: for any $a<b\in\bT$, we have 
\begin{equation}\label{e.densityomegat}
\tilde{\bbQ}_t(\omega_t\in [a,b])=\int_a^b \tilde{g}(t,x)dx.
\end{equation}
First, by definition we have $\tilde{\bbQ}_t^\eps(\omega_t\in[a,b])=\int_a^b \tilde{g}^\eps(t;t,x)dx$. Then, by Lemma~\ref{c.contildeg},  $\tilde{g}^\eps(t;t,\cdot)\to \tilde{g}(t,\cdot)$ in $C(\bT)$ in probability. Thus, upon extracting a subsequence, we have $\tilde{g}^\eps(t;t,\cdot)\to \tilde{g}(t,\cdot)$ in $C(\bT)$ almost surely. By Proposition~\ref{p.sdiff},  we have $\tilde{\bbQ}_t^\eps\Rightarrow \tilde{\bbQ}_t$ almost surely, which implies that, $\omega_t$ under $\tilde{\bbQ}_t^\eps$ converges in distribution to $\omega_t$ under $\tilde{\bbQ}_t$ (again, for almost every realization of $\zeta$). By the convergence of the density $\tilde{g}^\eps(t;t,\cdot)$ to the density $\tilde{g}(t,\cdot)$, this further implies that $\omega_t$ under $\tilde{\bbQ}_t$ has density given by $\tilde{g}(t,\cdot)$, hence completing the proof of \eqref{e.densityomegat}.

\emph{Step 2.} The goal is to show that for any $i=1,\ldots,n$,
\begin{equation}\label{e.conXi}
\begin{aligned}
(X_i^\eps,Y_i^\eps)&=(h^\eps(t,x_i)-h^\eps(t,0),\phi^\eps(t,x_i)-\phi^\eps(t,0))\\
&\to (h(t,x_i)-h(t,0),\bbQ_t(\omega_t\in[0,x_i])-x_i)=:(\bar{X}_i,\bar{Y}_i)
\end{aligned}
\end{equation}
in probability. Since $g^\eps=1+\nabla\phi^\eps$, we first write 
\[
\phi^\eps(t,x_i)-\phi^\eps(t,0)=\int_0^{x_i}g^\eps(t,y)dy-x_i.
\]
Recall that $g^\eps$ solves the Fokker-Planck equation \eqref{e.eqgeps1},
with $\cY_s^\eps$ the underlying diffusion given by \eqref{e.rwre1}. By the definition of $\bbQ_t^\eps$, we can write the integral on the r.h.s. of the above equation as 
\[
\int_0^{x_i}g^\eps(t,y)dy=\bbQ_t^\eps(\omega_t \in[0,x_i]).
\]
From \eqref{e.densityomegat}, we know that,  for any $x\in\bT$, 
\begin{equation}\label{e.claim1}
\tilde{\bbQ}_t(\omega_t=x)=0, 
\end{equation}
which implies that $\bbQ_t(\omega_t=x)=0$, because we have $\bbQ_t\stackrel{\text{law}}{=}\tilde{\bbQ}_t$ by Proposition~\ref{p.sdiff}. Thus, by the convergence of $\bbQ_t^\eps\Rightarrow \bbQ_t$, we can pass to the limit and obtain that, almost surely, 
\[
\begin{aligned}
\phi^\eps(t,x_i)-\phi^\eps(t,0)&=\bbQ_t^\eps(\omega_t \in[0,x_i])-x_i\\
&\to \bbQ_t(\omega_t\in[0,x_i])-x_i.
\end{aligned}
\]
This completes the proof of \eqref{e.conXi}.

\emph{Step 3.} The goal is to show that 
\begin{equation}\label{e.equallawXY}
(\bar{X}_1,\ldots,\bar{X}_n,\bar{Y}_1,\ldots,\bar{Y}_n)\stackrel{\text{law}}{=}(X_1,\ldots,X_n,Y_1,\ldots,Y_n),
\end{equation}
which is the last piece we need to complete the proof of \eqref{e.confinite}. To show \eqref{e.equallawXY}, we first rewrite $X_i,Y_i$ as  (by \eqref{e.defXepsYeps} and \eqref{e.deftildephi})
\begin{equation}\label{e.XYi}
\begin{aligned}
&X_i=-\beta W(x_i)=-\int_0^{x_i}u(0,y)dy, \\
&Y_i=\tilde{\phi}(t,x_i)=\int_0^{x_i}\tilde{g}(t,y)dy-x_i=\tilde{\bbQ}_t(\omega_t\in[0,x_i])-x_i,
\end{aligned}
\end{equation}
where the last ``$=$'' comes from \eqref{e.densityomegat}. Then we recall the definition of $\bar{X}_i,\bar{Y}_i$:
\begin{equation}\label{e.barXYi}
\begin{aligned}
&\bar{X}_i=h(t,x_i)-h(t,0)=\int_0^{x_i}u(t,y)dy,\\
&\bar{Y}_i=\bbQ_t(\omega_t\in[0,x_i])-x_i.
\end{aligned}
\end{equation}
If we compare \eqref{e.XYi} and \eqref{e.barXYi}, then \eqref{e.equallawXY} is a direct consequence of \eqref{e.identitylaw}.
\end{proof}

\section{A variance formula}
\label{s.var}

In this section, we combine the results obtained in the previous sections to derive a formula of $\sigma^2(\beta)$. Recall that in Section~\ref{s.co}, we have obtained the following expression
\begin{equation}\label{e.Xco1}
X_t=\int_{\R} x\rho(t,x)dx=\beta \int_0^t\int_{\bT} \EE[I_{1,t}(s,y)-I_{2,t}(s,y)|\F_s] \xi(s,y)dyds,
\end{equation}
with
\begin{equation}\label{e.i1i21}
I_{1,t}(s,y)-I_{2,t}(s,y)=\rho_{\rm m}(t,-|s,y)\int_{\bT}[\psi(s,y')-\psi(s,y)]\rho_{\rm m}(t,-|s,y')dy',
\end{equation}
and $\rho_{\rm m}(t,-|s,y)$ and $\psi(s,y)$ given by, cf \eqref{e.defrho} and \eqref{e.defW},
\begin{equation}\label{e.rhompsi}
\begin{aligned}
&\rho_{\rm m}(t,-|s,y)=\frac{ \G_{t,s}(-,y) \G_{s,0}(y,0)}{\G_{t,0}(-,0)}=\frac{\rhob(t,m_\bT;s,y)\rhof(s,y;0,0)}{\int_{\bT}\rhob(t,m_\bT;s,y')\rhof(s,y';0,0)dy'},\\
&\psi(s,y)=\frac{\sum_{n\in\Z} (n-y)\cZ_{s,0}(y,n)}{\G_{s,0}(y,0)}.
\end{aligned}
\end{equation}
Here we recall that $m_\bT$ denotes the Lebesgue measure on $\bT$. Obviously, the random function $\psi(s,\cdot)$ is $\F_s-$measurable. The factor $\rho_{\rm m}(t,-|s,y)$ depends on both the forward and backward polymer endpoint densities, and 
it is clear that $\rhof(s,y;0, 0)$  is $\F_s-$measurable and $\rhob(t,m_\bT;s,y)$ is independent of $\F_s$. In other words, to compute the conditional expectation on the r.h.s. of \eqref{e.Xco1}, we only need to average out the factor $\rhob$. 

This section contains a series of (rather standard) technical estimates which enable us to replace $\rho_{\rm m}(t,-|s,y)$ and $\psi(s,y)$ by their ``equilibrium'' versions. We briefly sketch them below so that the readers could have a big picture before going to the details. First, for $s\gg1$ and $t-s\gg1$, by the exponential mixing of the polymer endpoint density, one can replace the $\rho_{\rm f}$ and $\rho_{\rm b}$ in \eqref{e.rhompsi} by their stationary counterparts, and the $\rho_{\rm f}$ factor will only depends on $h(s,\cdot)-h(s,0)$, where we recall that $h$ solves the KPZ equation started at stationarity. Secondly, by Proposition~\ref{p.exerr}, one can replace $\psi(s,y')-\psi(s,y)$ by $\phi(s,y')-\phi(s,y)$. After these approximations, what will appear in the Clark-Ocone representation only involves $\nabla h(s,\cdot)$ and $\nabla\phi(s,\cdot)$, since we will first average out the $\rhob$ factor. Now we make use of the key identity  \eqref{e.equallaw} to compute the expectation of the resulting functional of $(\nabla h(s,\cdot),\nabla\phi(s,\cdot))$, and this leads to an explicit expression of $\sigma^2(\beta)$.

\subsection{Stationary approximation}
%
%
Recall that $h$ was defined in \eqref{e.defZW}, which is the solution to the KPZ equation started from stationary initial data $\beta W$. As we will have multiple independent Brownian bridges later, for notational convenience, we let $W_1=W$.
 Define the ``approximate'' midpoint density (as an approximation of $\rho_{\rm m}(t,-|s,y)$)
 \begin{equation}
 \begin{aligned}
 \rhoa(t,-|s,y)&=\frac{\rhob(t,\rho_2;s,y)e^{h(s,y)}}{\int_{\bT}\rhob(t,\rho_2;s,y')e^{h(s,y')}dy'}\\
 &=\frac{\rhob(t,\rho_2;s,y)\rhof(s,y;0,\rho_1)}{\int_{\bT}\rhob(t,\rho_2;s,y')\rhof(s,y';0,\rho_1)dy'},
 \end{aligned}
 \end{equation}
 with 
 \begin{equation}
   \label{rho-i}
 \rho_i(y)=\frac{e^{\beta W_i(y)}}{\int_{\bT}e^{\beta W_i(y')}dy'}, \quad i=1,2,
 \end{equation}
and $W_2$ is a Brownian bridge  independent of $\xi$ and $W_1$. 
Compare the expressions of $\rhoa$ and $\rho_{\rm m}$, the difference lies in the prescribed distribution of the starting and ending points. By the invariance of the law of $\rho_1,\rho_2$ under
the forward/backward polymer endpoint evolution, we actually have that for each $t\geq s>0$, 
 \begin{equation}\label{e.lawrhow}
 \{\rhoa(t,-|s,y)\}_{y\in\bT}\stackrel{\text{law}}{=}\left\{\frac{ e^{\beta (W_1(y)+W_2(y))}}{\int_{\bT}e^{\beta (W_1(y')+W_2(y'))dy'}}\right\}_{y\in\bT},
 \end{equation}
 where, by the elementary properties of the Brownian bridge, the r.h.s. is statistically invariant under rotation on a torus (see e.g. \cite[Lemma 4.2]{ADYGTK22}).

  For $s\gg1$ and $t-s\gg1$, by the exponential mixing of the endpoint
  distribution of directed polymers, see \cite[Theorem 2.3]{GK21},
  one may expect that $\rho_{\rm m}(t,-|s,y)$, see 
  \eqref{e.rhompsi}, is close to $\rhoa(t,-|s,y)$.  
   This is the contents of the following lemma.
 \begin{lemma}\label{l.rhotilderho}
For any $p\in[1,\infty)$, there exist $C,\lambda>0$ such that 
\begin{equation}
  \label{061705-23}
\sup_{y\in\bT} \E_{W_1}\E_{W_2}\EE |\rho_{\rm m}(t,-|s,y)-\rhoa(t,-|s,y)|^p \leq C(e^{-\lambda(t-s)}+ e^{-\lambda s})
\end{equation}
for all $t\geq 100$ and $s\in[1,t-1]$.
\end{lemma}

 The proof is rather standard given the exponential mixing property
 formulated in Proposition~\ref{p.tk1} below. We postpone  it till the
 Appendix \ref{secA.1}.
  
 Define 
 \begin{equation}\label{e.defJt}
 J_t(s,y):=\rhoa(t,-|s,y)\int_{\bT}[\phi(s,y')-\phi(s,y)]\rhoa(t,-|s,y')dy',
 \end{equation}
 and
 \begin{equation}
 \label{Yt}
 Y_t:=\beta\int_0^t \int_{\bT}\E_{W_2}\EE[J_t(s,y)|\F_s]\xi(s,y)dyds.
\end{equation}
The $J_t(s,y)$ should be viewed as a ``stationary''
approximation of $I_{1,t}(s,y)-I_{2,t}(s,y)$, with $\rho_{\rm m}$ replaced by
$\rhoa$ and $\psi$ by $\phi$.   We emphasize that
$Y_t$ only depends on the noise $\xi$ and the Brownian bridge $W_1$.
The  It\^o integral with respect to $\xi$ in \eqref{Yt} is
performed for a  fixed realization of $W_1$.

The main result of the present section is the following lemma, which implies that $\tfrac{X_t-Y_t}{\sqrt{t}}\to0$ as $t\to\infty$:
\begin{lemma}\label{l.XYt}
 We have
$\sup_{t\ge 100}\E_{W_1}\EE |X_t-Y_t|^2 <\infty$.
\end{lemma}
\begin{proof}
 \blue{First, we note that $X_t$ does not depend on $W_1$.} Fix any $t\geq100$. The proof consists of several steps.

\emph{Step 1, uniform bounds}. By \eqref{e.psiphiW} and Corollary \ref{cor011205-23}, we have for any $p\geq1$, there exits a constant $C>0$ such that 
\begin{equation}\label{e.511}
\begin{aligned}
\sup_{s\geq0,y,y'\in\bT}\left\{\EE |\psi(s,y')-\psi(s,y)|^p+\E_{W_1}\EE |\phi(s,y')-\phi(s,y)|^p\right\} \leq C.
\end{aligned}
\end{equation}
All constants $C$ appearing below may depend on $p\in[1,+\infty)$, but
are independent of other parameters such as $t,s$ etc.
In addition, by \eqref{e.lawrhow}, we have 
\begin{equation}\label{e.bdtilderho}
\sup_{s\in[0,t],y\in\bT}\E_{W_1}\E_{W_2}\EE\rhoa(t,-|s,y)^p\leq C.
\end{equation}
 Combining estimates \eqref{e.511} and \eqref{e.bdtilderho}, we conclude that
\begin{equation}\label{e.514}
\sup_{s\in(0,t],y\in\bT} \E_{W_1}\E_{W_2}\EE |J_t(s,y)|^p \leq C.
\end{equation}

Next, we derive an estimate on $I_{1,t}-I_{2,t}$. First, for $\rho_{\rm m}(t,-|s,y)$, 
starting from the expression in \eqref{e.rhompsi}, by the H\"older estimate together with \eqref{e.mmbdrho} and
\eqref{e.mmbdrho-a}, we conclude that
\begin{align*}
\EE \rho_{\rm m}(t,-|s,y)^p \leq C, \quad\quad \mbox{ for } s\in[1,t-1],y\in\bT.
\end{align*}
For other values of $s$, we will use the following estimates: (i) for $s\in (t-1,t]$, we have  
\begin{equation}
  \label{011305-23}
\sup_{s\in(t-1,t],y\in\bT} \EE[\rhob(t,m_\bT;s,y)^p +\rhob(t,m_\bT;s,y)^{-p}] \leq C,
\end{equation}
and (ii) for $s\in(0,1)$, we have
\begin{equation}
  \label{021305-23}
  \begin{aligned}
     \EE [\rhof(s,y;0,0)^p] \leq CG_s(y)^p,\quad\quad \EE [\rhof(s,y;0,0)^{-p}] \leq CG_s(y)^{-p}.
\end{aligned}
\end{equation}
Here we recall that $G$ is the heat kernel on torus. 

 Suppose now that $s\in[0,1)$. Then we have
\[
\rho_{\rm m}(t,-|s,y)\le \rho_{\rm f}(s,y;0,0) \sup_{y'\in\bT}\rho_{\rm
  b}(t,m_{\bT};s,y') \sup_{y'\in\bT}\rho^{-1}_{\rm
  b}(t,m_{\bT};s,y').
\]
On the other hand, when $s\in(t-1,t]$, then 
\[
\rho_{\rm m}(t,-|s,y)\le \rho_{\rm f}(s,y;0,0) \sup_{y'\in\bT}\rho_{\rm
  b}(t,m_{\bT};s,y') \sup_{y'\in\bT} \rho_{\rm f}^{-1}(s,y';0,0).
\]
Using again H\"older inequality 
 together with \eqref{011305-23} and
\eqref{021305-23}, we
  derive that 
\begin{equation}\label{e.513}
\EE \rho_{\rm m}(t,-|s,y)^p \leq
C(G_s(y)^p1_{s\in(0,1)}+1_{s\in(1,t]}), \quad\quad 0< s\le t,\,  y\in\bT.
\end{equation}
Hence,
from \eqref{e.i1i21}, \eqref{e.511} and \eqref{e.513},  we
conclude that, for all $s\in(0,t], y\in\bT$, 
  \begin{equation}\label{e.514-a}
\EE| I_{1,t}(s,y)-I_{2,t}(s,y)|^p  \leq
C\big(G_s(y)^p1_{s\in(0,1)}+1_{s\in(1,t]}\big).
\end{equation}


\emph{Step 2, approximation.}  We claim that there exist $C,\lambda>0$
such that
\begin{equation}\label{e.515}
\sup_{y\in\bT}\E_{W_1}\E_{W_2}\EE |I_{1,t}(s,y)-I_{2,t}(s,y)-J_t(s,y)|^p \leq C\big(e^{-\lambda s}+e^{-\lambda(t-s)}\big)
\end{equation}
for  all  $s\in[1,t-1]$.

Indeed, using \eqref{e.i1i21} and \eqref{e.defJt}, we can decompose the difference as
$$
I_{1,t}(s,y)-I_{2,t}(s,y)-J_t(s,y)=\sum_{i=1}^3 E_i,
$$ with 
\[
\begin{aligned}
&E_1=\rho_{\rm m}(t,-|s,y)\int_{\bT}[\psi(s,y')-\psi(s,y)-\phi(s,y')+\phi(s,y)]\rho_{\rm m}(t,-|s,y')dy',\\
&E_2=(\rho_{\rm m}(t,-|s,y)-\rhoa(t,-;s,y))\int_{\bT}[\phi(s,y')-\phi(s,y)]\rho_{\rm m}(t,-|s,y')dy',\\
&E_3=\rhoa(t,-|s,y)\int_{\bT}[\phi(s,y')-\phi(s,y)](\rho_{\rm m}(t,-|s,y')-\rhoa(t,-|s,y'))dy'.
\end{aligned}
\]
Now we apply Proposition~\ref{p.exerr} and  \eqref{e.513} 
to conclude that for $s\in[1,t-1]$,
\[
  \E_{W_1}\EE|E_1|^p\leq C e^{-\lambda s},
\]
From Lemma~\ref{l.rhotilderho}, \eqref{e.511} and  \eqref{e.513},  we
obtain that 
\[
  \E_{W_1}\E_{W_2}\EE |E_2|^p \leq C(e^{-\lambda(t-s)}+e^{-\lambda
    s}),
  \]
Finally from Lemma~\ref{l.rhotilderho}, 
\eqref{e.511} and \eqref{e.bdtilderho}, we get
\[
  \E_{W_1}\E_{W_2}\EE |E_3|^p \leq C(e^{-\lambda(t-s)}+e^{-\lambda s}).
\]
This completes the proof of \eqref{e.515}.

\emph{Step 3, conclusion.} We have (see \eqref{e.Xco1} and  \eqref{Yt})
\[
X_t-Y_t=\beta \int_0^t \int_{\bT}\E_{W_2}\EE [I_{1,t}(s,y)-I_{2,t}(s,y)-J_t(s,y)|\F_s]\xi(s,y)dyds.
\]
For each realization of $W_1$, applying It\^o isometry, we have 
\[
\EE |X_t-Y_t|^2=\beta^2\int_0^t\int_{\bT} \EE \left[\big( \E_{W_2}\EE [I_{1,t}(s,y)-I_{2,t}(s,y)-J_t(s,y)|\F_s]\big)^2\right] dyds.
\]
Taking expectation on $W_1$ and using the Jensen inequality, we  obtain 
\[
\E_{W_1}\EE|X_t-Y_t|^2 \leq \beta^2 \int_0^t\int_{\bT} \E_{W_1}\E_{W_2}\EE \big[|I_{1,t}(s,y)-I_{2,t}(s,y)-J_t(s,y)|^2\big] dyds.
\]
Applying \eqref{e.514} and \eqref{e.514-a} to the integral in $s\in(0,1]\cup [t-1,t]$ and \eqref{e.515} to the integral in $s\in[1,t-1]$, we complete the proof.
\end{proof}

\subsection{It\^o isometry and time reversal} By Lemma~\ref{l.XYt}, to
derive the limit of $t^{-1}\EE X_t^2$, it suffices to do the same for
$\E_{W_1}\EE Y_t^2/t$. Recall that $Y_t$ and $J_t(s,y)$ are given by
\eqref{Yt} and \eqref{e.defJt} respectively  and  $W_1$, $W_2$ are
  two  
independent Brownian bridges satisfying  
$W_j(0)=W_j(1)=0$, $j=1,2$.

 Before applying It\^o isometry to compute the variance, there is one more simplification we need to deal with the conditional expectation. 
 Define 
\begin{equation}\label{e.defrhow}
\rho_{W}(s,y):=\frac{e^{\beta W_2(y)+h(s,y)}}{\int_{\bT} e^{\beta W_2(y')+h(s,y')}dy'}=\frac{e^{\beta W_2(y)+h(s,y)-h(s,0)}}{\int_{\bT} e^{\beta W_2(y')+h(s,y')-h(s,0)}dy'},
\end{equation}
where we recall that $h(s,y) =\log Z_{W_1}(s,y)$ (cf \eqref{e.defZW}),
and define
\begin{equation}\label{e.defJW}
J_{W}(s,y):=\rho_{W}(s,y)\int_{\bT}[\phi(s,y')-\phi(s,y)]\rho_{W}(s,y')dy',
\end{equation}
where $\phi$ is defined by \eqref{phi}:
\[
\phi(t,x)=\frac{\int_{\R} \cZ_{t,0}(x,x')(x'-x)e^{\beta W_1(x')}dx'}{\int_{\R} \cZ_{t,0}(x,x') e^{\beta W_1(x')}dx'}.
\]
We claim that
\begin{equation}\label{e.516}
\E_{W_2}\EE [J_t(s,y)|\F_s]=\E_{W_2}J_{W}(s,y).
\end{equation}
First, since  
both $h(s,\cdot)$ and $\phi(s,\cdot)$ are $\F_s-$measurable, we 
note that the right hand side is also $\F_s-$measurable. Secondly, recall that 
\[
\rhoa(t,-|s,y)=\frac{\rhob(t,\rho_2;s,y)e^{h(s,y)}}{\int_{\bT}\rhob(t,\rho_2;s,y')e^{h(s,y')}dy'},
\]
and we have, for any $t>s$ fixed, 
\[
\left\{\rhob(t,\rho_2;s,y)\right\}_{y\in\bT}\stackrel{\text{law}}{=}\left\{\frac{e^{\beta W_2(y)}}{\int_{\bT}e^{\beta W_2(y')}dy'}\right\}_{y\in\bT}.
\]
 Since $\rhob(t,\rho_2;s,y)$ depends only on the noise
$\{\xi(\ell,x)\}_{(\ell,x)\in[s,t]\times \R}$ and the bridge $W_2$,  we can
write that
\[
\begin{aligned}
&\E_{W_2}\EE [\rhoa(t,-|s,y) \rhoa(t,-|s,y') |\F_s]\\
&=\E_{W_2}[\rho_{W}(s,y)
\rho_{W}(s,y')],\quad y,y'\in\bT.
\end{aligned}
\]
Due to $\F_s-$measurability of 
both $h(s,\cdot)$ and $\phi(s,\cdot)$, the above
  yields \eqref{e.516}, which further implies 
  \[
  Y_t=\beta\int_0^t \int_{\bT}\E_{W_2}[J_W(s,y)]\xi(s,y)dyds.
  \]

By It\^o isometry (applied for each realization of $W_1$), we have 
\begin{equation}\label{e.ito1}
\E_{W_1}\EE Y_t^2=\beta^2\int_0^t\int_{\bT} \E_{W_1}\EE
\left[\Big(\E_{W_2} [J_{W}(s,y)]\Big)^2\right] dyds.
\end{equation}
From the expression of $J_W(s,y)$ in \eqref{e.defJW}, we
see that in order to compute the  expectation on the r.h.s., one
needs to understand the   distribution of
\[
\Big(h(s,\cdot)-h(s,0),\phi(s,\cdot)-\phi(s,0)\Big).
\] This is where the time reversal anti-symmetry (see \eqref{e.burgersAsym}) and Proposition~\ref{p.keyP} come into play. Define 
\begin{equation}\label{e.tildeJ}
\tilde{J}_{W}(s,y):=\tilde{\rho}_{W}(y)\int_{\bT}[\tilde{\phi}(s,y')-\tilde{\phi}(s,y)]\tilde{\rho}_{W}(y')dy',
\end{equation} 
where $\tilde{\phi}$ was defined in \eqref{e.deftildephi} and 
\begin{equation}\label{e.tilderho}
\tilde{\rho}_{W}(y):=\frac{ e^{\beta W_2(y)} e^{-\beta W_1(y)}}{\int_{\bT}e^{\beta W_2(y')}e^{-\beta W_1(y')}dy'}.
\end{equation}
We have the following key proposition:
 \begin{proposition}\label{p.Appreversal}
  For any $s>0, y\in\bT$, we have
 \begin{equation}
    \label{011605-23}
    J_{W}(s,y) \stackrel{\rm law}{=}\tilde{J}_{W}(s,y).
  \end{equation}
  In consequence,
 \begin{equation}
    \label{031705-23}
\E_{W_1}\EE \big[\left(\E_{W_2} [J_{W}(s,y)]\right)^2\big]=\E_{W_1}\EE \big[\left(\E_{W_2} [\tilde{J}_{W}(s,y)]\right)^2\big].
\end{equation}
\end{proposition} 

\begin{proof}
By \eqref{e.defrhow} and \eqref{e.defJW}, one can write 
\begin{equation}\label{e.661}
\begin{aligned}
J_W(s,y)=&\frac{e^{\beta W_2(y)+h(s,y)-h(s,0)}}{\int_{\bT} e^{\beta W_2(y')+h(s,y')-h(s,0)}dy'}\\
&\times \int_{\bT}[\phi(s,y')-\phi(s,y)]\frac{e^{\beta W_2(y')+h(s,y')-h(s,0)}}{\int_{\bT} e^{\beta W_2(y'')+h(s,y'')-h(s,0)}dy''}dy'.
\end{aligned}
\end{equation}
By \eqref{e.equallaw}, we know that 
\[
(h(s,\cdot)-h(s,0),\phi(s,\cdot)-\phi(s,0))\stackrel{\text{law}}{=}(-\beta W_1(\cdot),\tilde{\phi}(s,\cdot)).
\]
If we replace $h(s,y)-h(s,0)$ by $-\beta W_1(\cdot)$ and $\phi$ by $\tilde{\phi}$ in \eqref{e.661}, we have $J_W(s,y)$ becomes $\tilde{J}_W(s,y)$. This completes the proof.
\end{proof}

\subsection{Stationary approximation again} In the expression of $\tilde{J}_W$ defined in \eqref{e.tildeJ}, there is still an $s-$dependence through the function $\tilde{\phi}(s,\cdot)$.  The last step is to use a stationary approximation of $\tilde{\phi}(s,\cdot)$ for large $s\gg1$. By the definition of $\tilde{\phi}$ in \eqref{e.deftildephi}, we have
\[
\tilde{\phi}(s,y')-\tilde{\phi}(s,y)=\int_y^{y'} [\tilde{g}(s,z)-1]dz,
\]
with (cf \eqref{e.deftildeg})
\[
\tilde{g}(s,z)=\int_{\bT} \left(\frac{\rhob(s, x;0,z)e^{\beta
      W_1(z)}}{\int_\bT\rhob(s, x;0,z')e^{\beta W_1(z')}dz'}\right) dx , \quad\quad z\in \bT.
\]
For $s\gg1$ and any $x\in\bT$, we expect the backward polymer endpoint density $\rhob(s,x;0,\cdot)$ to stabilize and reach
stationarity.

This inspires us to define
\begin{equation}\label{e.defginfinity}
\tilde{g}_{\mathrm{app}}(s,z):=\frac{\rhob(s,\rho_3;0,z)e^{\beta W_1(z)}}{\int_{\bT} \rhob(s,\rho_3;0,z')e^{\beta W_1(z')}dz'}, \quad\quad s\geq0, z\in\bT,
\end{equation}
where
$$
\rho_3(z)=\frac{e^{\beta W_3(z)}}{\int_{\bT} e^{\beta W_3(z')}dz'},
$$ and $W_3$ is another Brownian bridge independent of
$W_1,W_2,\xi$.

By the same argument as used in the proof of
Lemma~\ref{l.rhotilderho}, see Appendix \ref{secA.1}, we have: for any $p\in[1,\infty)$, there exist
$C,\lambda>0$ such that  
\begin{equation}
  \label{011705-23}
\sup_{z\in\bT}\E_{W_1}\E_{W_3}\EE|\tilde{g}(s,z)-\tildegapp(s,z)|^p \leq C e^{-\lambda s}, \quad\quad s\geq1.
\end{equation}

Again, by the invariance of the law of $\rho_3$ under
the backward polymer endpoint evolution, it is clear that for each
$s\geq0$, we have 
  \begin{equation}\label{e.laststa}
     \begin{aligned}
       &\left\{\Big(\tildegapp(s,z),
\tilde{\rho}_{W}(z')\Big)\right\}_{(z,z')\in\bT^2} \stackrel{\text{law}}{=}
\left\{\Big(\tildegapp(0,z) ,
\tilde{\rho}_{W}(z')\Big)\right\}_{(z,z')\in\bT^2}\\
&
=\left\{\left(\frac{e^{\beta W_3(z)+\beta W_1(z)}}{\int_{\bT} e^{\beta W_3(z'')+\beta W_1(z'')}dz''}, \frac{ e^{\beta W_2(z')} e^{-\beta W_1(z')}}{\int_{\bT}e^{\beta W_2(z'')}e^{-\beta W_1(z'')}dz''}\right)\right\}_{(z,z')\in\bT^2}.
\end{aligned}
\end{equation}

Define 
\begin{equation}\label{e.defJinfinity}
\tilde{J}_{W,\mathrm{app}}(s,y):=\tilde{\rho}_{W}(y)\int_{\bT}\left(\int_y^{y'}[\tildegapp(s,z)-1]dz\right)\tilde{\rho}_{W}(y')dy'.
\end{equation}
By \eqref{e.laststa}, we know that 
\begin{equation}\label{e.66inlaw}
\left\{\tilde{J}_{W,\mathrm{app}}(s,y)\right\}_{y\in\bT}\stackrel{\text{law}}{=}\left\{\tilde{J}_{W,\mathrm{app}}(0,y)\right\}_{y\in\bT},
\end{equation}
and from the definitions of $\tilde{\rho}_{W}$  and $\tilde{g}_{\mathrm{app}}$, it is clear that $\tilde{J}_{W,\mathrm{app}}(0,\cdot)$ only involves the three independent Brownian bridges $\{W_i\}_{i=1,2,3}$.

The following lemma is the last piece we need to compute the variance:
\begin{lemma}\label{l.AppSta}
There exist $C,\lambda>0$ such that
\[
\sup_{y\in\bT} \E_{W}\EE\big[ |
\tilde{J}_{W}(s,y)-\tilde{J}_{W,\mathrm{app}}(s,y)|^2\big]\leq C
e^{-\lambda s},\quad s\ge1.
\]
\end{lemma}

\begin{proof}
From the expressions of $\tilde{J}_{W}(s,y)$ and $\tildeJapp(s,y)$ in \eqref{e.tildeJ} and \eqref{e.defJinfinity}, we have 
\[
\begin{aligned}
&\tilde{J}_{W}(s,y)-\tildeJapp(s,y)\\
&=\tilde{\rho}_{W}(y)\int_{\bT}\left(\int_y^{y'}[\tilde{g}(s,z)-\tildegapp(s,z)]dz\right)\tilde{\rho}_{W}(y')dy'.
\end{aligned}
\]
It suffices to apply \eqref{011705-23} together with the H\"older inequality to complete the proof.
\end{proof}

A corollary of the above lemma is 
\begin{corollary}
We have 
\begin{equation}
  \label{021705-23}
\lim_{t\to\infty} \frac{\EE X_t^2}{t}=\lim_{t\to\infty} \frac{\E_{W_1}\EE Y_t^2}{t}=\beta^2 \int_{\bT}\E_{W_1} \E_{W_3} \Big[\big(\E_{W_2} [\tildeJapp(0,y)]\big)^2\Big] dy.
\end{equation}
\end{corollary}

\begin{proof}
Recall that we already know from \cite[Section 5]{YGTK22} that $\EE
X_t^2/t$ converges as $t\to\infty$, so, the first  equality in
\eqref{021705-23}  is a consequence of Lemma~\ref{l.XYt}. For the
second  equality, we have from \eqref{e.ito1} and \eqref{031705-23}
\[
\begin{aligned}
\E_{W_1} \EE Y_t^2&=\beta^2\int_0^t\int_{\bT} \E_{W_1}\EE \big[\left(\E_{W_2} [J_{W}(s,y)]\right)^2\big] dyds\\
&=\beta^2\int_0^t\int_{\bT} \E_{W_1}\EE \big[\left(\E_{W_2} [\tilde{J}_{W}(s,y)]\right)^2\big]dyds.
\end{aligned}
\]
Applying Lemma~\ref{l.AppSta} and \eqref{e.66inlaw}, we complete the proof.
\end{proof}

By \eqref{e.varcon}, \eqref{e.exqvar} and the above corollary, we
conclude that for any $\beta>0$ we have
\begin{equation}\label{e.varre1}
\sigma^2(\beta)= 1+\beta^2 \int_{\bT}\E_{W_1}\E_{W_3} \left[\big(\E_{W_2} [\tildeJapp(0,y)]\big)^2\right] dy.
\end{equation}

To write the effective diffusivity more explicitly, we have the following lemma, using which we complete the proof of Theorem~\ref{t.mainth}.
\begin{lemma}
The process $\{\tildeJapp(0,y)\}_{y\in\bT}$ is stationary. As a result, we have 
\begin{equation}
  \label{041705-23}
\sigma^2(\beta)=1+\beta^2\E_{W_1}\E_{W_3}\big[\left(\E_{W_2} [\tildeJapp(0,0)]\right)^2\big],
\end{equation}
where
\[
 \E_{W_2}\tilde{J}_{W,\mathrm{app}}(0,0)
    =\int_{\bT^2} \Xi(\beta,y,W_1)
    \left(\frac{e^{\beta W_1(z)+\beta W_3(z)}}{\int_{\bT}
    e^{\beta W_1(z')+\beta
      W_3(z')}dz'}-1\right)1_{[0,y]}(z)dydz,
      \]
      with
      \[
\Xi(\beta,y,W_1)=\E_{W_2}\left[\frac{ e^{\beta W_2(y) -\beta
    W_1(y)}}{\Big(\int_{\bT}e^{\beta W_2(y')}e^{-\beta W_1(y')}dy'\Big)^2}\right].
\]
\end{lemma}
\proof
By \eqref{e.defginfinity}, we know that $\tildegapp(s,\cdot)$ is a continuous density on $\bT$, so 
\[
y'\mapsto \int_{y}^{y'}[\tildegapp(s,z)-1]dz
\]
is a $1$-periodic function for each $y\in\bT$.  For any $x\in\bT$, we can write
\begin{align*}
  &\tilde{J}_{W,\mathrm{app}}(0,y+x)=\tilde{\rho}_{W}(y+x)\int_{\bT}\left(\int_{y+x}^{y'}[\tildegapp(0,z)-1]dz\right)\tilde{\rho}_{W}(y')dy'\\
  &
    =\tilde{\rho}_{W}(y+x)\int_{\bT}\left(\int_{y+x}^{y'+x}[\tildegapp(0,z)-1]dz\right)\tilde{\rho}_{W}(y'+x)dy'\\
  &
    =\tilde{\rho}_{W}(y+x)\int_{\bT}\left(\int_{y}^{y'}[\tildegapp(0,z+x)-1]dz\right)\tilde{\rho}_{W}(y'+x)dy',
\end{align*}
and the conclusion of the lemma follows from the joint stationarity of 
\[
\left\{(\tildegapp(0,z) ,
\tilde{\rho}_{W}(z'))\right\}_{(z,z')\in\bT^2},
\] see \eqref{e.laststa}.
\qed

\appendix

\section{Endpoint distribution of the directed polymer on a cylinder}

\subsection{Exponential mixing of endpoint distribution}

Here we summarize a few results on the exponential mixing of the endpoint distribution of the directed polymer on a torus, which will be used throughout the paper. Recall that $\rhof,\rhob$ were defined in \eqref{e.forwardbackward}. By the time reversal invariance of the spacetime white noise, we have 
\begin{equation}
  \label{e.mmbdrho-a}
\{\rhof(t,x;s,\nu)\}_{x\in\bT}\stackrel{\text{law}}{=}\{\rhob(t,\nu;s,x)\}_{x\in\bT}.
\end{equation}
Here is the main result on $\rhof$. By the above identity, the same conclusion applies to $\rhob$.
\begin{proposition}
\label{p.tk1}
The Markov process $\{\rhof(t,\cdot;s,\nu)\}_{t\geq0}$, taking values
in ${\cal M}_1(\bT)$, has a unique invariant measure given by
the law of the ${\cal M}_1(\bT)$-valued random variable $\rho(y)dy$,
where
\begin{equation}
   \label{rho-i}
 \rho(y)=\frac{e^{\beta W(y)}}{\int_{\bT}e^{\beta W(y')}dy'},
 \end{equation}
 and $W$ is a Brownian bridge satisfying $W(0)=W(1)=0$.

 Furthermore, for any $p\geq 1$, there exist $C,\lambda>0$ such that for all $t\geq1$, 
\begin{equation}
  \label{051705-23}
\EE \sup_{\nu,\nu'\in\mathcal{M}_1(\bT)} \sup_{x\in\bT} |\rhof(t,x;0,\nu)-\rhof(t,x;0,\nu')|^p  \leq Ce^{-\lambda t}, 
\end{equation}
and
\begin{equation}\label{e.mmbdrho}
\EE \sup_{\nu\in \mathcal{M}_1(\bT)}\sup_{x\in\bT}\, \{\rhof(t,x;0,\nu)^p
+\rhof(t,x;0,\nu)^{-p}\} \leq C.
\end{equation} 
\end{proposition}
\proof
The fact that $\{\rhof(t,\cdot;s,\nu)\}_{t\geq0}$  is Markovian
  follows from  \cite[Lemma 2.2]{GK21}.  Estimate \eqref{051705-23} is an immediate consequence of  \cite[Proposition 4.7]{GK21}.

To show \eqref{e.mmbdrho}, we note that
\begin{align*}
 & \rhof(t,x;0,\nu)=\frac{\int_{\bT}\G_{t,t-1/2}(x,x') \rhof
  (t-1/2,x';0,\nu)dx'}{\int_{\bT^2}\G_{t,t-1/2}(x'',x')  \rhof
  (t-1/2,x';0,\nu)dx'dx''}\\
  &
    \le \sup_{x,x'\in\bT}\G_{t,t-1/2}(x,x')  \sup_{x',x''\in\bT}\G_{t,t-1/2}^{-1}(x'',x').
\end{align*}
Estimate \eqref{e.mmbdrho}  then follows from an application of
\cite[Lemma 4.1]{GK21} and the H\"older inequality.
\qed

\subsection{Proof of Lemma~\ref{l.rhotilderho}} 
\label{secA.1}
Recall that the goal was to show that for $s\in[1,t-1]$, 
\begin{equation}\label{e.66app1}
\sup_{y\in\bT} \E_{W_1}\E_{W_2}\EE |\rho_{\rm m}(t,-|s,y)-\rhoa(t,-|s,y)|^p \leq C(e^{-\lambda(t-s)}+ e^{-\lambda s}),
\end{equation}
with
\[
\begin{aligned}
&\rho_{\rm m}(t,-|s,y)=\frac{\rhob(t,m_\bT;s,y)\rhof(s,y;0,0)}{\int_{\bT}\rhob(t,m_\bT;s,y')\rhof(s,y';0,0)dy'},\\
&\rhoa(t,-|s,y)=\frac{\rhob(t,\rho_2;s,y)\rhof(s,y;0,\rho_1)}{\int_{\bT}\rhob(t,\rho_2;s,y')\rhof(s,y';0,\rho_1)dy'}.
\end{aligned}
\]
By \eqref{e.mmbdrho-a} and \eqref{051705-23}, we have 
\[
\begin{aligned}
&\E_{W_2}\EE \left[\sup_{y\in\bT} |\rhob(t,m_\bT;s,y)-\rhob(t,\rho_2;s,y)|^p\right]\\
&+\E_{W_1}\EE \left[\sup_{y\in\bT}|\rhof(s,y;0,0)-\rhof(s,y;0,\rho_1)|^p\right] \leq C\big(e^{-\lambda(t-s)}+e^{-\lambda s}\big).
\end{aligned}
\]
With the above estimate, the rest of the proof is rather standard, with several uses of H\"older inequality together with \eqref{e.mmbdrho}. We do not repeat it here.

\section{Singular diffusion: proof of Proposition~\ref{p.sdiff}}
\label{s.sd}
Let us first recall the main ideas in Section~\ref{s.scalar} to see the role played by singular diffusion.  
The main argument in Section~\ref{s.scalar} was done on the level of approximation for each $\eps>0$ and only passed to the limit through Proposition~\ref{p.sdiff} -- in the following we will sketch formally the argument for $\eps=0$.

Recall that the goal was to describe the joint distribution of $h(t,\cdot)-h(t,0)$ and $\phi(t,\cdot)-\phi(t,0)$, where $h$ solves the KPZ equation at stationarity, and $\phi(t,x)$, defined in \eqref{e.defphi}, is the displacement of the polymer started at $(t,x)$, running backward in time, with the terminal potential $e^{\beta W(\cdot)}$. By the Girsanov theorem and Feynman-Kac representation, one can interpret $\phi(t,x)$ as the quenched mean displacement of the diffusion in the Burgers drift:
\[
\phi(t,x)=\E_B \cX_t-x,
\]
where $\{\cX_s\}_{s\in[0,t]}$ solves the following formal SDE 
\begin{equation}\label{e.formalsde1}
d\cX_s=u(t-s,\cX_s)ds+dB_s,\quad\quad \cX_0=x.
\end{equation}
Note that \eqref{e.formalsde1}  is only symbolic, since $u=\nabla h$ is the distributional-valued solution to the stochastic Burgers equation. In addition, one can show that $\phi$ solves the following PDE with distribution-valued coefficients
\begin{equation}
\partial_t\phi=\frac12\Delta\phi+u\nabla \phi+u, \quad\quad\phi(0,x)=0,
\end{equation}
which is also symbolic. To study the increment process of $\phi$, it suffices to consider $g=1+\nabla\phi$, which solves the Fokker-Planck equation with a random distribution-valued coefficient
\begin{equation}\label{e.formalfk1}
\partial_s g(s,y)=\frac12\Delta_y g(s,y)+\nabla_y (u(s,y)g(s,y)), \quad\quad g(0,x)=1.
\end{equation}
Since $\nabla h=u$, to study the joint distribution of $h(t,\cdot)-h(t,0)$ and $\phi(t,\cdot)-\phi(t,0)$, the problem reduces to studying the joint distribution of $u(t,\cdot)$ and $g(t,\cdot)$. The above Fokker-Planck equation is related to another diffusion in the Burgers flow:
\begin{equation}\label{e.formalsde2}
d\cY_s=-u(s,\cY_s)ds+dB_s, \quad\quad \cY_0\sim m_\bT.
\end{equation}
Inspired by the time-reversal anti-symmetry of the Burgers flow \eqref{e.burgersAsym}, we consider $\tilde{\cY}_s$ given by 
\begin{equation}\label{e.formalsde3}
\d\tilde{\cY}_s=u(t-s,\tilde{\cY}_s)ds+dB_s,\quad\quad \tilde{\cY}_0\sim m_\bT,
\end{equation}
and its density, denoted by $\tilde{g}(t;s,\cdot)$, evolves according to 
\begin{equation}\label{e.formalfk2}
\partial_s \tilde{g}(t;s,y)=\frac12\Delta_y \tilde{g}(t;s,y)-\nabla_y(u(t-s,y)\tilde{g}(t;s,y)), \quad\quad \tilde{g}(t;0,y)=1.
\end{equation}
Compare the two (formal) Fokker-Planck equations \eqref{e.formalfk1} and \eqref{e.formalfk2}, since the coefficients have the same law, it is natural to expect the solutions to have the same law, and, the joint distributions of the coefficient and the solution are the same as well. This turns out to be precisely the \eqref{e.identitylaw} in the statement of Proposition~\ref{p.sdiff}:
\[
(\mathscr{U}_t,\bbQ_t)\stackrel{\text{law}}{=}(\tilde{\mathscr{U}}_t,\tilde{\bbQ}_t),
\]
where $\mathscr{U}_t,\tilde{\mathscr{U}}_t$ are the corresponding coefficient processes, and $\bbQ_t,\tilde{\bbQ}_t$ are the probability measures corresponding to $g$ and $\tilde{g}$. 

The singular diffusions described by the formal SDE \eqref{e.formalsde1}, \eqref{e.formalsde2} and \eqref{e.formalsde3} can all be made sense pathwisely, that is, for each realization of $u$, as  the solutions to the corresponding martingale problems \cite{CC18,DD16}, using the tools of rough path, regularity structures and paracontrolled calculus. The singular Fokker-Planck equations \eqref{e.formalfk1} and \eqref{e.formalfk2} can also be made sense pathwisely. As these are not the focus of the paper, we refrain from going to the details and only refer the readers to the aforementioned papers. Our proof of Proposition~\ref{p.sdiff} is based on \cite{DD16}, which was the first one to give a rigorous meaning of the so-called continuum directed random polymer, introduced in \cite{AKQ14}, as a singular diffusion.

\begin{proof}[Proof of Proposition~\ref{p.sdiff}]
Fix $t>0$, recall that $\bbQ_t^\eps, \tilde{\bbQ}_t^\eps$ are the quenched probability measures on $C([0,t],\bT)$ of 
\[
\begin{aligned}
&d\cY_s^\eps=-u^\eps(s,\cY_s^\eps)ds+dB_s, \quad\quad \cY_0^\eps \sim m_\bT,\\
&d\tilde{\cY}_s^\eps=u^\eps(t-s,\tilde{\cY}_s^\eps)ds+dB_s, \quad\quad \tilde{\cY}_0^\eps\sim m_\bT.
\end{aligned}
\]
For any $x\in\bT$, we denote $\bbQ_{t,x}^\eps,\tilde{\bbQ}_{t,x}^\eps$ as the corresponding measures when the starting point is $\cY_0^\eps=\tilde{\cY}_0^\eps=x$. Applying \cite[Theorem 31]{DD16}, we know that, for almost every realization $\zeta$, $\tilde{\bbQ}_{t,x}^\eps$ converges weakly\footnote{The result in \cite[Theorem 31]{DD16} is only for the solution to the KPZ equation with constant initial data (rather than at stationarity), but the same proof applies verbatim to our setting.}, with the limit denoted by $\tilde{\bbQ}_{t,x}$. Furthermore, this can be done for all $x\in\bT$ with the same $\zeta$. As a result, for almost every realization $\zeta$, we have the convergence of $\tilde{\bbQ}_t^\eps=\int_{\bT} \tilde{\bbQ}_{t,x}^\eps dx$. 

To show the convergence of $\bbQ_t^\eps$, one could either follow  the proof in \cite{DD16}, which is essentially a repetition, or follow the paracontrolled approach outlined in \cite{CC18}. Since the proofs are almost the same, we do not do it here.

In the end, one needs to show the identity-in-law:
\[
(\mathscr{U}_t,\bbQ_t)\stackrel{\text{law}}{=}(\tilde{\mathscr{U}}_t,\tilde{\bbQ}_t).
\]
This immediately comes from the pathwise construction: $\bbQ_t$ and $\tilde{\bbQ}_t$ are constructed for almost every realization of $\mathscr{U}_t=\{-u(s,x)\}_{s\in[0,t],x\in\bT}$ and $\tilde{\mathscr{U}}_t=\{u(t-s,x)\}_{s\in[0,t],x\in\bT}$ respectively. In other words, they are deterministic functionals of $\mathscr{U}_t$ and $\tilde{\mathscr{U}}_t$ respectively.  By the fact that $\mathscr{U}_t\stackrel{\text{law}}{=}\tilde{\mathscr{U}}_t$, we complete the proof.
\end{proof}


\begin{thebibliography}{99}


  
  \bibitem{AKQ14}
  {\sc T.~Alberts, K.~Khanin and J.~Quastel}, {\em The continuum directed random polymer}, Journal of Statistical Physics 154.1-2 (2014): 305-326.


\bibitem{ABK24}
{\sc S.~Armstrong, A.~Bou-Rabee and T.~Kuusi}, {\em Superdiffusive central limit theorem for a Brownian particle in a critically-correlated incompressible random drift}, arXiv preprint arXiv:2404.01115 (2024).

\bibitem{BL16}
{\sc J.~Baik and Z.~Liu}, {\em T{ASEP} on a ring in sub-relaxation time scale},
  J. Stat. Phys., 165 (2016), 1051--1085.

\bibitem{BL18}
{\sc J.~Baik and Z.~Liu}, {\em Fluctuations of {TASEP} on a ring in relaxation
  time scale}, Comm. Pure Appl. Math., 71 (2018), 747--813.

\bibitem{BL19}
{\sc J.~Baik and Z.~Liu}, {\em Multipoint distribution of periodic {TASEP}}, J.
  Amer. Math. Soc., 32 (2019), 609--674.

\bibitem{BL21}
{\sc J.~Baik and Z.~Liu}, {\em Periodic {TASEP} with general initial
  conditions}, Probab. Theory Related Fields, 179 (2021), 1047--1144.

\bibitem{BLS20}
{\sc J.~Baik, Z.~Liu, and G.~L.~F. Silva}, {\em {Limiting one-point
  distribution of periodic TASEP}}, Aug. 2020, arXiv preprint
  \href{https://arxiv.org/abs/2008.07024v1}{2008.07024v1}.

%
\bibitem{BQS11}
{\sc M.~Bal{\'a}zs, J.~Quastel, and T.~Sepp{\"a}l{\"a}inen}, {\em Fluctuation
  exponent of the {KPZ}/stochastic {B}urgers equation}, J. Amer. Math. Soc., 24
  (2011), 683--708.
%
%
%

\bibitem{BCY23}
{\sc  G.~Barraquand, I.~Corwin and Z.~Yang},  {\em Stationary measures for integrable polymers on a strip}, arXiv preprint arXiv:2306.05983 (2023).

\bibitem{BD21}
{\sc G.~Barraquand and P.~L. Doussal}, {\em Steady state of the KPZ equation
  on an interval and Liouville quantum mechanics}, Europhysics Letters 137.6 (2022): 61003.
%
  
  \bibitem{BC95}
  {\sc L.~Bertini and N. Cancrini}, {\em The Stochastic Heat Equation: Feynman-Kac Formula and Intermittence}, 
  J. Stat. Phys., 78 (1995), 1377--1401.

%
%
%
%
%
  \bibitem{Bru03}
  {\sc {\'E}.~Brunet}, {\em Fluctuations of the winding number of a directed polymer in a random medium}, Physical Review E 68.4 (2003): 041101.

\bibitem{BD00}
{\sc {\'E}.~Brunet and B.~Derrida}, {\em Ground state energy of a non-integer
  number of particles with $\delta$ attractive interactions}, Phys. A, 279
  (2000), 398--407.
  
  \bibitem{BD001}
{\sc {\'E}.~Brunet and B.~Derrida}, {\em Probability distribution of the free energy of a directed polymer in a random medium}, Physical Review E 61.6 (2000): 6789.

  \bibitem{GK211}
{\sc {\'E}.~Brunet, Y.~Gu and T.~Komorowski}, {\em {High temperature behaviors of the directed polymer on a cylinder}}, arxiv preprint arXiv:2110.07368v3.



\bibitem{BKWW21}
{\sc W.~Bryc, A.~Kuznetsov, Y.~Wang, and J.~Wesolowski}, {\em {Markov processes
  related to the stationary measure for the open KPZ equation}}, July 2021,
  arXiv preprint \href{https://arxiv.org/abs/2105.03946v2}{2105.03946v2}.
  
\bibitem{CC18}  
 {\sc G.~Cannizzaro and K.~Chouk}, {\em Multidimensional SDEs with singular drift and universal construction of the polymer measure with white noise potential}, The Annals of Probability 46.3 (2018): 1710-1763.
 
 \bibitem{CLF22}
{\sc G.~Cannizzaro, L.~Haunschmid-Sibitz and F.~Toninelli}, {\em $\sqrt{\log t}$-Superdiffusivity for a Brownian particle in the curl of the 2D GFF}, The Annals of Probability 50.6 (2022): 2475-2498.
  


%
  
  \bibitem{CMOW22}
  {\sc G.~Chatzigeorgiou, P.~Morfe, F.~Otto and L.~Wang},  {\em The Gaussian free-field as a stream function: asymptotics of effective diffusivity in infra-red cut-off}, arXiv preprint arXiv:2212.14244.

\bibitem{CKNP19}
{\sc L.~Chen, D.~Khoshnevisan, D.~Nualart, and F.~Pu}, {\em Spatial
  ergodicity for SPDEs via Poincar\'e-type inequalities}
 July 2019, arXiv preprint
  \href{https://arxiv.org/abs/1907.11553v1}{1907.11553v1}.

%


%
  
  
%
%

\bibitem{Cor12}
{\sc I.~Corwin}, {\em The {K}ardar-{P}arisi-{Z}hang equation and universality
  class}, Random Matrices Theory Appl., 1 (2012), 1130001, 76.

\bibitem{Cor22}  
 {\sc I.~Corwin}, {\em Some recent progress on the stationary measure for the open KPZ equation},  Toeplitz Operators and Random Matrices: In Memory of Harold Widom (2022): 321-360.

\bibitem{CK21}
{\sc I.~Corwin and A.~Knizel}, {\em {Stationary measure for the open KPZ
  equation}}, Apr. 2021, arXiv preprint
  \href{https://arxiv.org/abs/2103.12253v2}{2103.12253v2}.

%

\bibitem{DD16}
{\sc F.~Delarue and R.~Diel},  {\em Rough paths and 1d SDE with a time dependent distributional drift: application to polymers}, Probability Theory and Related Fields 165.1-2 (2016): 1-63.


\bibitem{DEM93}
{\sc B.~Derrida, M.~R. Evans and D.~Mukamel}, {\em Exact diffusion constant
  for one-dimensional asymmetric exclusion models}, J. Phys. A, 26 (1993),
  4911--4918.

\bibitem{DEM95}
{\sc B.~Derrida, M.~R. Evans and K.~Mallick}, {\em Exact diffusion constant of a one-dimensional asymmetric exclusion model with open boundaries}, Journal of statistical physics 79 (1995): 833-874. 

\bibitem{DM97}
{\sc B.~Derrida and K.~Mallick}, {\em Exact diffusion constant for the one-dimensional partially asymmetric exclusion model}, Journal of Physics A: mathematical and general 30.4 (1997): 1031.

  

%
%
  

  
  
    \bibitem{ADYGTK22} {\sc A.~Dunlap, Y.~Gu and T.~Komorowski}, {\em
        Fluctuation exponents of the KPZ equation on a large torus},
      {\em Comm. on Pure and Applied Math.}, 2023, https://doi.org/10.1002/cpa.22110


\bibitem{DGL23}  
{\sc A.~Dunlap, Y.~Gu and L.~Li}, {\em Localization length of the 1+ 1 continuum directed random polymer},  Annales Henri Poincar\'e. Vol. 24. No. 7. Cham: Springer International Publishing, 2023.



\bibitem{AF98}
{\sc A.~Fannjiang}, {\em Anomalous diffusion in random flows},  Mathematics of Multiscale Materials (1998): 81-99.

\bibitem{BF22}
{\sc B.~Fehrman}, {\em Stochastic homogenization with space-time ergodic divergence-free drift},  arXiv preprint arXiv:2207.14555 (2022).

\bibitem{FH22}
{\sc G.~Feltes and H.~Weber}, {\em Brownian particle in the curl of 2-d stochastic heat equations}, arXiv preprint arXiv:2211.02194 (2022).



  
  
    \bibitem{YGTK22}  {\sc Y.~Gu and T.~Komorowski}, {\em     Fluctuations of the winding
      number of a directed polymer on a cylinder},  to appear in {\em
   SIAM Journ. Math. Anal.} available at
 https://arxiv.org/abs/2207.14091.

\bibitem{GK21}
{\sc Y.~Gu and T.~Komorowski}, {\em {KPZ on torus: Gaussian fluctuations}},
  June 2021, arXiv preprint
  \href{https://arxiv.org/abs/2104.13540v2}{2104.13540v2}, to appear in {\em Ann. Inst. H. Poincare,
     Prob. and Stat.}
  
  


%
  
  \bibitem{GIP15}
{\sc  M.~Gubinelli, P.~Imkeller and N.~Perkowski}, {\em Paracontrolled distributions and singular PDEs}, Forum of Mathematics, Pi. Vol. 3. Cambridge University Press, 2015.
  
  \bibitem{MH13}
{\sc M.~Hairer}, {\em Solving the KPZ equation}, Annals of mathematics (2013): 559-664.
  
  \bibitem{MH14}
  {\sc M.~Hairer}, {\em A theory of regularity structures}, Inventiones mathematicae 198.2 (2014): 269-504.

  
  \bibitem{HZZZ21}
  {\sc Z.~Hao, X.~Zhang, R.~Zhu and X.~Zhu}, {\em Singular kinetic equations and applications},  arXiv preprint arXiv:2108.05042.
  
  
\bibitem{HK22}
{\sc  Y.~Hu and K.~L\^e}, {\em Asymptotics of the density of parabolic Anderson random fields},  Annales de l'Institut Henri Poincare (B) Probabilites et statistiques. Vol. 58. No. 1. Institut Henri Poincar\'e, 2022.


\bibitem{DK14}
{\sc D.~Khoshnevisan}, {\em Analysis of stochastic partial differential equations}, Vol. 119. American Mathematical Soc., 2014.

 \bibitem{KLO12}  {\sc T.~Komorowski, C.~Landim and S.~Olla}, {\em
Fluctuations in Markov Processes. Time Symmetry and Martingale Approximation},
Springer Ser.: {\em Grundlehren der mathematischen Wissenschaften}, Vol. {\bf 345}, 2012.




\bibitem{KP22}
{\sc H.~Kremp  and N.~Perkowski}, {\em Multidimensional SDE with
  distributional drift and L\'evy noise}, Bernoulli 28.3 (2022):
1757-1783.

\bibitem{LQSY04}
{\sc C.~Landim, J.~Quastel, M.~Salmhofer and H.~T.~Yau}, {\em Superdiffusivity of asymmetric exclusion process in dimensions one and two}, Communications in mathematical physics, 244, (2004) 455-481.





  \bibitem{Liu18}
  {\sc Z.~Liu}, {\em Height fluctuations of stationary TASEP on a ring in relaxation time scale}, 
  Ann. Inst. H. Poincar\'e B, 54 (2018), 1031--1057. 

%

%
%

\bibitem{Nua06}
{\sc D.~Nualart}, {\em The Malliavin calculus and related topics}, Vol. 1995. Berlin: Springer, 2006.


%
%


\bibitem{Qua12}
{\sc J.~Quastel}, {\em Introduction to {KPZ}}, in Current Developments in
  Mathematics, 2011, Int. Press, Somerville, MA, 2012, 125--194.


\bibitem{QS15}
{\sc J.~Quastel and H.~Spohn}, {\em The one-dimensional {KPZ} equation and its
  universality class}, J. Stat. Phys., 160 (2015), 965--984.
  
  \bibitem{Sep20}
 {\sc T.~Sepp\"al\"ainen}, {\em Existence, uniqueness and coalescence of directed planar
geodesics: Proof via the increment-stationary growth process}, Ann. Inst. H.
Poincare Probab. Statist. 56(3): 1775-1791, 2020.
  
  \bibitem{Yau04}
{\sc  H.~T.~Yau}, {\em $(\log t)^{2/3}$ Law of the Two Dimensional Asymmetric Simple Exclusion Process}, Annals of mathematics (2004): 377-405.

%
%
%
%
%


\end{thebibliography}
 \end{document}